\documentclass[11pt]{amsart}
\usepackage{graphicx}
\usepackage{latexsym}
\usepackage{amsfonts,amsmath,amssymb}
\newtheorem{theorem}{Theorem}[section]

\newtheorem{lemma}[theorem]{Lemma}

\newtheorem{corollary}[theorem]{Corollary}

\usepackage{color}

\def\ex{\textup{E\/}}

\def\eps{\varepsilon}
\def\la{\lambda}
\def\a{\alpha}
\def\be{\beta}

\def\ga{\gamma}
\def\part{\partial}

\newcommand{\beq}{\begin{equation}}
\newcommand{\eeq}{\end{equation}}

\theoremstyle{remark}

\numberwithin{equation}{section}
\date{\today}
\begin{document}
\title[Random pairing model]{On Bollob\'as-Riordan random pairing model of preferential attachment graph. }
\date{}
\author{Boris Pittel}
\address{Department of Mathematics, The Ohio State University, Columbus, OH 43210, USA.}
\email{bgp@math.osu.edu}

\begin{abstract} Bollob\'as-Riordan random pairing model of a preferential attachment graph $G_m^n$ is studied. Let $\{W_j\}_{j\le mn+1}$ be the process of sums of independent exponentials with mean $1$. We prove that the degrees of the first $\nu_m^n:=n^{\frac{m}{m+2}-\eps}$ vertices are jointly, and uniformly, asymptotic to $\{2(mn)^{1/2}\bigl(W^{1/2}_{mj}-W^{1/2}_{m(j-1)}\bigr)\}_{j\in [\nu_m^n]}$, and that with high
probability (whp)  the smallest of these degrees is $n^{\frac{\eps(m+2)}{2m}}$, at least. Next we bound the probability that there exists a pair of large vertex sets without connecting edges,
and apply the bound to several special cases. We propose to measure an influence of a vertex $v$ by the size of a maximal recursive tree (max-tree) rooted at $v$. We show that whp the set of the first $\nu_m^n$ vertices does not contain a max-tree,
and the largest max-tree has size of order $n$. We 
prove that, for $m>1$, $\Bbb P(G_m^n\text{ is connected})\ge 1- O\bigl((\log n)^{-(m-1)/3+o(1)}\bigr)$. We show that  the distribution of scaled size
of a generic max-tree in $G_1^n$ converges to a mixture of two beta distributions.
\end{abstract}


\maketitle

\section{Definitions,  main results}

In  $1999$ Barab\'asi and Albert \cite{BarAlb} proposed a dynamic model of a growing network in which a newcomer vertex
attaches itself to the older vertices with probability distribution strongly favoring the vertices of higher degrees. This paper opened
the floodgate of research which continues unabated twenty years later. For historical accounts and rigorous results we refer the reader 
to Bollob\'as \cite{Bol}, Bollob\'as and Riordan \cite{BolRio},  \cite{BolRio3}, \cite{BolRio4}, Bollob\'as, Riordan, Spencer and
Tusn\'adi \cite{Bol1}, Bollob\'as, Borgs, Chayes, Riordan \cite{BolBorChaRio}, Cooper and Frieze \cite{CooFri}, M\'ori \cite{Mor}, Katona and M\'ori \cite{KatMor}, Pek\"oz, R\"ollin and Ross \cite{PekRolRos},  Pittel \cite{Pit2}, Acan and Hitczenko \cite{AcaHit}, Berger, Borgs, Chayes and Saberi \cite{BerBorChaSab}, Frieze and Karo\'nski
\cite{FriKar}, van der Hofstad \cite{Hof} and  Frieze, P\'erez-Gim\'enez, Pra\l{}at and Reiniger \cite{FGPR}, to name the work the author is most aware of. 

We focus on a rigorously defined  random graph process $\{G_m^t\}_{t=1}^{\infty}$ with preferential attachment introduced and studied by Bollob\'as and Riordan \cite{BolRio}. Liberally citing \cite{BolRio}, consider $m=1$ first. The graphs are nested: $G_1^1\subset G_1^2\subset\cdots$, with $G_1^t$ having 
vertex set $[t]$ and $t$ edges/loops, and with precisely $1$ edge/loop incident to vertex $t$. Therefore for $s<t$ the graph
$G_1^s$ is a subgraph of $G_1^t$ induced by $[s]$. In particular, $G_1^1$ has a single vertex $1$ and a single loop. A loop contributes $2$ to a vertex degree.  
It is postulated that, given $G_1^{t-1}$,  vertex $t$ attaches by an edge  to a vertex $s\in [t-1]$ ($t$ develops a loop resp.) with probability proportional to the degree of $s$ in $G_1^{t-1}$ (with probability proportional to $1$ resp.). For $m>1$,  $m$ edges
emanating from $t$ are added, one at a time, to $G_m^{t-1}$, at each of $m$ steps ``counting the previous edges as well as the 'outward
half' of the edge being added as already contributing to the degrees''. They demonstrated that to  get the graph $G_m^n$, one takes the graph $G_1^{mn}$ and considers the consecutive $m$-long blocks of vertices of $G_1^{mn}$ as vertices $1,\dots,n$.

 
The authors of \cite{BolRio} discovered that this process can be gainfully described in terms of {\it linearized chord diagrams} 
LCD, a rich combinatorial scheme from the enumerative knots theory. Earlier Stoimenov \cite{Sto}, Bollob\'as and Riordan \cite{BolRio*} 
and Zagier \cite{Zag} used the LCDs
to upper bound the total number of independent Vassiliev invariants of a given degree. 
An LCD with $n$ chords consists of $2n$ distinct points on the $x$-axis matched by semi-circular chords in the upper half plane. There are $(2n-1)!!$ such matchings. Given a diagram $L$ we associate with $L$ a graph $\phi(L)$ with $n$ vertices as follows. There are $n$ right endpoints of the $n$ arcs. Vertex $1$ of
$\phi(L)$ consists of all endpoints starting from the leftmost point up to and including its match, i. e. the first right endpoint. Vertex $2$ consists of all subsequent endpoints all the way up to and including the second right endpoint, and so on, up to the last, $n$-th, right endpoint.
$\phi(L)$ has an edge $(i,j)$, $i< j$, if and only if there is an arc with the left endpoint and the right endpoint from the point sets
associated with vertex $i$ and vertex $j$ respectively. There is a (single) loop at vertex $i$ if the right endpoint, i. e. the rightmost point of the set associated with $i$, is matched with another point of the set. 

The authors of \cite{BolRio} claimed and proved, by induction, that if $L$ is chosen uniformly at random from all $(2n-1)!!$ LCDs with $n$ chords then $\phi(L)$ has the same distribution as $G_1^n$. This is crucial, since as an option it allows to study $G_1^n$ directly, ``without going through the process''.  
The random graph $\phi(L)(\overset{\mathcal D}\equiv G_1^{N})$, $N=mn$, was used in \cite{BolRio} to prove a deep result on a likely upper bound for the diameter of $G_m^n$. A key idea was to view the uniformly random pairing  of $2N$ points as being induced by the sequence of $2N$
independent, $[0,1]$-uniform random variables $X_i$. We pair $X_{2i-1}$ with $X_{2i}$ and then relabel all $X_i$ in the ascending order. 
The resulting pairing of the elements of $[2N]$ is uniformly random on the set of all $(2N-1)!!$ pairings. Observe that
$r_i\boldsymbol{\boldsymbol{:}}=\max\{X_{2i-1},X_{2i}\}$ has density $2x$, and conditioned on $r_i$, $\ell_i\boldsymbol{\boldsymbol{:}}=\min\{X_{2i-1},X_{2i}\}$ is distributed uniformly on
$[0,r_i]$. Let $R_1,\dots,R_N$ be the values of $r_1,\dots,r_N$ sorted in the increasing order. By the discussion of $\phi(L)$, its structure
is determined by the number of left endpoints in each of the intervals $(R_{i-1},R_i]$, $0< i\le N$, $R_0=0$. 

In the preliminary comments the authors of \cite{BolRio} discussed persistent technical  complications arising from non-uniformity of the right endpoints distribution. It occurred to us that  the right endpoints
become exactly uniform if the random variables $x_i$ have density $f(x)=\frac{1}{2\sqrt{x}}$. In this case, conditioned on $r_i=\max\{x_{2i-1},x_{2i}\}$, $\ell_i=\min\{x_{2i-1}, x_{2i}\}$ has density $\frac{1}{2\sqrt{r_ix}}$, $x\in [0,r_i]$. The advantage of having $r_i$ uniform is that the order statistics $R_1,\dots, R_N$ can be described quite explicitly. Introduce the sequence of independent exponentials 
 $w_i$, $\Bbb P(w_i\ge x)=e^{-x}$, $x\ge 0$. Let $W_j=\sum_{i\in [j]} w_i$; then 
 \[
\bold R\boldsymbol{\boldsymbol{:}}= \{R_j\}_{j\in [mn]}\overset{\mathcal D}\equiv\boldsymbol{\mathcal R}\boldsymbol{:}=\{W_j/W_{mn+1}\}_{j\in [mn]},
 \]
see Karlin and Taylor \cite{KarTay} or R\'enyi \cite{Ren}. (We stumbled on these sources back in $1988$ while studying the random stable matchings, \cite{Pit1}.) Remarkably, Bollob\'as and Riordan \cite{BolRio3}, \cite{BolRio4} (Thm 17) had found that  the density $f(x)=\frac{1}{2\sqrt{x}}$ whence independent exponentials and their sums, arise {\it asymptotically\/} in their pairing model with the $[0,1]$-Uniforms $x_i$ as well. See also Frieze and Karo\'nski \cite{FriKar}, Exer. 17.4.5. 

Our aim in this paper is to contribute to analysis of the Bollob\'as-Riordan model $G_m^n$ based on the $f(x)$-modification of their random sequence $\{ x_i\}$. A main thread running through the proofs is that, conditioned on $\bold W=\{W_j\}$, the edge-indicators (the "no/edge"
indicators resp.) are either independent or negatively independent, which leads to the Chernoff-type and the product-type bounds for the
conditional probabilities, amenable to asymptotic estimates based on the properties of the process $\bold W$. 
 
 In Section 2.1 (Theorem \ref{degs,m}), we prove that  the degrees of the first $\nu:=n^{\frac{m}{m+2}-\eps}$ vertices are jointly, and uniformly, asymptotic to $\{2(mn)^{1/2}$\linebreak $\bigl(W^{1/2}_{mj}-W^{1/2}_{m(j-1)}\bigr)\}_{j\in [\nu]}$. (For the finite $\nu$ the convergence with rates was established in \cite{PekRolRos}; an alternative proof without rates was given in \cite{AcaHit}. Years earlier a closely related result on the limiting behavior of the maximum degree was stated, with a proof sketch,  in \cite{BolRio4} (Thm 17).)
It follows that with high probability (whp) the first $n^{\frac{m}{m+2}-\eps}$ vertex degrees are each at least $n^{\frac{\eps(m+2)}{2m}}$. In contrast (Theorem \ref{second}), the degrees of the vertices from the interval, say
 $[0.01n, 0.99 n]$ whp are all of logarithmic order, at most.
 
In Section 2.3 we prove that whp the multigraph $G_m^n$ is almost simple:  the total number of loops is asymptotic to $\frac{m+1}{4}\log n$  (Theorem \ref{Ln}) and the total number of pairs of parallel edges is asymptotic to $\frac{m^2-1}{16}\log^2 n$ (Theorem \ref{mathcalP<}). The key to the proof is that, conditioned on $\bold W$, the
numbers of loops at distinct vertices are independent, and the numbers of parallel edges joining distinct pairs of vertices are negatively
associated. (It was proved in \cite{BolRio4} that the expected number of triangles in $G_m^n$ is asymptotic to $\frac{m(m^2-1)}{48} \log^3n$.)
 
Following the  lead of \cite{FGPR}, in Section 3 (Theorem \ref{Emu,nu<}) we bound the probability that there exists a pair of large vertex sets with no edges joining them.  We use the bound to show that if $m$ is large then whp there are no such pairs with each set of cardinality $\gtrsim \frac{4\log m}{m} n$. In \cite{FGPR} whp non-existence was proved for two sets each of cardinality $\gtrsim \frac{16\log m}{m} n$. We also use
Theorem \ref{Emu,nu<} to show that, for $m$
large, {\bf (a)\/} with probability $>1-\exp(-\a(\rho) n)$  every vertex set of cardinality strictly between $n f(\rho)\left(\frac{\log m}{m}\right)^2$ and $\frac{n}{1+\rho}(1-m^{-1}g(\rho))$
is vertex-expanding with rate $\ge \rho$, and {\bf (b)\/} for $\ga>0$, with probability $> 1-\exp(-\be(\ga)n)$ there are no isolated sets of cardinality between $\phi(\ga)\left(\frac{\log m}{m}\right)^2$ and $(\ga+2)^{-1}$.
 
In Section 4 
we analyze the sizes of recursive
subtrees in $G_m^n$; ``recursive'' means that the vertices increase along every path emanating from the oldest vertex of a subtree. The size of a maximal recursive tree rooted at a given vertex $v$ can be viewed, we think, as an influence measure of $v$. We show (Theorem \ref{maxtree})
that whp there are no maximal recursive trees with a vertex set chosen from the first $\nu:=n^{\frac{m}{m+4}-\eps}$ vertices. More generally, for $\mu=o(n)$,
whp no subset of $[\mu]$ of cardinality comparable to $\mu$ is a vertex set of a maximal recursive tree. We show that with positive
limiting probability $G_m^n$ contains a spanning recursive tree, and  whp the size of the largest recursive tree is of order $n$. In Section 4.1 we prove that for $m>1$, with probability $1- O\bigl((\log n)^{-(m-1)/3+o(1)}\bigr)$,  the graph $G_m^n$  is connected..

We conclude this section with the Chernoff-type inequalities, which we use in the proofs. (For the first two Lemmas see \cite[Thms 2.1, 2.8 ] {JLR}.)

\begin{lemma}\label{LemmaA}
Let $X$ be a sum of independent Bernoulli random variables $X_i$ with probabilities $p_1,\dots,p_n$. Denote $\mu =\ex[X]=\sum_i p_i$. Then
\begin{align}
&\Bbb P(X\le \mu-t)  \le e^{-t^2/(2\mu)}  \qquad (t>0), \label{C1} \\
&\Bbb P(|X-\mu| \ge \eps\mu)  \le 2e^{-\eps^2\mu/3} \qquad (0<\eps\le 3/2).		\label{C3}
\end{align}
\end{lemma}

\begin{lemma}\label{LemmaB}
\begin{align}
\Bbb P(X\ge \mu+t)&\le \exp\left(\!-\mu\phi\left(\frac{t}{\mu}\right)\!\right), \quad t\ge 0,\label{X>mu}\\
\Bbb P(X\le \mu-t)&\le \exp\left(\!-\mu\phi\left(\frac{-t}{\mu}\right)\!\right), \quad t\in [0,\mu],\label{X<mu}\\
\phi(x)&\boldsymbol{\boldsymbol{:}}=(1+x)\log(1+x)-x.
\end{align}
\end{lemma}

\begin{lemma}\label{note} The bound \eqref{X>mu}, for any $\mu\ge \ex[X]$, holds also for the negatively associated Bernoulli random variables.
\end{lemma}
\begin{proof}
Indeed
\begin{align*}
\ex\left[\binom{X}{k}\right]&=\sum_{i_1<\cdots<i_k}\Bbb P(X_{i_1}=\cdots =X_{i_k}=1)\le\sum_{i_1<\cdots<i_k}\prod_{j=1}^k \Bbb P(X_{i_j}=1)
\\
&\qquad\quad\le \frac{1}{k!}\Bigl(\sum_{j\in [n]} p_j\Bigr)^k=\frac{\ex^k[X]}{k!}\le\frac{\mu^k}{k!},
\end{align*}
implying that, for $z>1$,
\[
\ex\bigl[z^X\bigr]=\sum_{k=0}^n\ex\left[\binom{X}{k}\right](z-1)^k\le \sum_{k=0}^n \frac{\mu^k}{k!}(z-1)^k\le e^{\mu(z-1)}.
\]
Therefore
\[
\Bbb P(X\ge \mu+t)\le \inf_{z>1}\frac{\ex\bigl[z^X\bigr]}{z^{\mu+t}}\le \inf_{z>1} \frac{e^{\mu(z-1)}}{z^{\mu+t}}=\exp\left(\!-\mu\phi\left(\frac{t}{\mu}\right)\!\right).
\]
\end{proof}
We will also need Chernoff-type inequalities for the sums of $w_i$.

\begin{lemma}\label{W} Let $W_{\nu}=\sum_{a=1}^{\nu} w_a$, where $\{w_a\}$ are  independent exponentials with $\ex[w]=1$. Then, (1) for every
$\a>1$,
\begin{equation}\label{C4}
\Bbb P( W_{\nu}\ge \a\nu)\le\exp\bigl(-\nu\phi(\a)\bigr),\quad \phi(z)=z-\log z -1;
\end{equation}
(2) for every $\a<1$,
\begin{equation}\label{C4'}
\Bbb P( W_{\nu}\le \a\nu)\le\exp\bigl(-\nu\phi(\a)\bigr).
\end{equation}
Consequently, for $\sigma\in (0,1/2)$,
\begin{equation}\label{C4''}
\Bbb P\Bigl(|W_{\nu}-\nu|\le \nu^{1-\sigma},\,\,\forall\,\nu\ge N)\ge 1-\exp\bigl(-\Theta(N^{1-2\sigma})\bigr).
\end{equation}
\end{lemma}
 
 \begin{lemma} \label{W1} Let $V_{\mu}=\sum\limits_{k\in [\mu]} d_{k}w_{k}$, $d_{k}\ge 0$. Then for  every $\a\in (0,1)$
\begin{equation}\label{C5}
\Bbb P\Bigl(V_{\mu}\le (1-\a)\sum_{k\in [\mu]}d_k\Bigr)\le \exp\left(-\frac{\a^2}{2}\cdot\frac{\left(\sum_k d_k\right)^2}{\sum_kd_k^2}\right).
\end{equation}
\end{lemma}
\begin{proof}  Let $z>0$. The probability in question is at most
\begin{align*}
&\exp\bigl(z(1-\alpha)\ex[V_{\mu}]\bigr)\cdot\ex\bigl[\exp(-z V_{\mu})\bigr]=\prod_{k\in [n]} \exp(z(1-\alpha)d_k)(1+zd_k)^{-1}\\
&\le\! \exp\!\Biggl(\sum_k \!\left(\!z(1-\a)d_k -zd_k+\frac{z^2d_k^2}{2}\right)\!\!\Biggr)=\exp\Biggl(\!-z\a\sum_k d_k+\frac{z^2}{2}\sum_k d_k^2\Biggr),
\end{align*}
and the last exponent attains its minimum at $z=\a(\sum_k d_k)/(\sum_k d_k^2)$.
\end{proof}
\section{Vertex degrees.} 


We generate $G_1^{mn}$ as follows. Start with the sequence $\{w_j\}_{j\in [mn+1]}$ of independent exponentials and
introduce the sums $W_j=\sum_{i\in [j]} w_i$. Define the sequence $\bold R=\{R_k\}_{k\in [mn]}$, where $R_k=W_k/W_{mn+1}$. Each $R_k$
is the right endpoint of the pair $(\ell_k, R_k)$, and conditioned on $\bold W$, the variables $\ell_k$ are independent, with densities
$\frac{1}{2\sqrt{R_kx}}$, $x\in (0,R_k]$. The resulting LCD has the same distribution as $G^1_{mn}$. As mentioned in Introduction, the graph $G_m^n$ with $n$ vertices is obtained from $G_1^{mn}$ by forming $m$-long consecutive blocks of vertices of $G_1^{mn}$. 

Let $D(j)$ denote the degree of vertex $j$ in $G_m^n$. Then
\begin{equation}\label{Dmn(j)=}
D(j)=d(j)+\sum_{i>mj}\Bbb I\bigl(R_{m(j-1)}<\ell_i<R_{mj}\bigr),\quad j\in [n],\,\,(R_0\boldsymbol{\boldsymbol{:}}=0).
\end{equation}
The term ``$d(j)$'' on the RHS of \eqref{Dmn(j)=} accounts for the $m$ chords joining the $m$ right endpoints $R_{m(j-1)+1},\dots, R_{mj}$, 
belonging to vertex $j$, with their respective left endpoints $\ell_{m(j-1)+1},\dots, \ell_{mj}$. We get a loop at $j$ each time when a left endpoint happens to belong to vertex $j$; so $d(j)\in [m,2m]$. The sum is the total count of edges that join the left
endpoints from the $j$-th vertex in $G_m^n$ with the respective right endpoints belonging to the subsequent vertices $j+1,\dots,n$.
\subsection{Joint degree distribution of the first $n^{\frac{m}{m+2}-\eps}$ vertices.}
\begin{theorem} \label{degs,m} Let
\begin{align*}
&\qquad\mathcal D(j)=
2(mn)^{1/2}\bigl(W_{mj}^{1/2}-W_{m(j-1)}^{1/2}\bigr),\\
&\qquad\qquad j_n=\lfloor n^a\rfloor,\quad a\in (0,m/(m+2)), \\
&\delta <[1-a(m+2)/m]/4,\quad b<(1-a(m+2)/m)/2.
\end{align*}
 Then {\bf (1)\/} 
\begin{align*}
&\Bbb P\left(D(j)=\bigl(1+O(n^{-\delta})\bigr)\mathcal D(j);\,\,\forall\,j\le j_n\right)\ge 1-n^{-\Delta},\\
&\quad\forall\,\Delta<\min\Bigl(\frac{1-a}{2}-\delta;\,\frac{m}{2}\Bigl(1-\frac{a(m+2)}{m}\Bigr)-2m\delta\Bigr).
\end{align*}
{\bf (2)\/} Consequently 
\begin{align*}
&\quad\,\,\Bbb P\Bigl(\min_{j\le j_n}D(j)>n^b\Bigr)\ge 1-n^{-\Delta_1},\\
&\forall\,\Delta_1<\min\left(\Delta;\,m\bigl[\bigl(1-\tfrac{a(m+2)}{m}\bigr)/2-b\bigr]\right).
\end{align*}
Thus whp the first $n^{m/(m+2)-\eps}$ vertex degrees are each at least $n^{\eps(m+2)/2m}$.
\end{theorem}
The part {\bf (1)} asserts that the degrees of the first $j_n$ vertices in $G_m^n$ are uniformly asymptotic to the increments of the process
$\bigl\{2(mn)^{1/2} W_{mj}^{1/2}\boldsymbol{:}\,j\in [j_n]\bigr\}$.

\begin{proof} {\bf (1)\/} For each given $j$, the indicators in equation \eqref{Dmn(j)=}, {\it conditioned on $\bold W=\{W_k\}$\/}, are independent. 
So we anticipate that, conditionally, $D^*(j):=D(j)-d(j)$ (in-degree of $j$) is sharply concentrated around $\mu_j(\bold W):=\ex\bigl[D^*(j)\boldsymbol|\,\bold W]$. Now
\begin{align}
&\mu_j(\bold W)=\sum_{i>m j}\int_{R_{m(j-1)}}^{R_{mj}}\frac{dx}{2(R_i x)^{1/2}}\notag\\
&=\sum_{i> mj}\frac{R^{1/2}_{mj}-R^{1/2}_{m(j-1)}}{R^{1/2}_i}= \sum_{i>m j}
\frac{W^{1/2}_{mj}-W^{1/2}_{m(j-1)}}{W^{1/2}_i}\notag\\
&=\frac{\Omega_{mj}}{W^{1/2}_{mj}+W^{1/2}_{m(j-1)}}\sum_{i> mj} W_i^{-1/2},\quad \Omega_{mj}\boldsymbol{\boldsymbol{:}}=\!\!\sum_{t=m(j-1)+1}^{mj}\!\!\!\!w_t. \label{EDmn(j)=}
\end{align}
Let us estimate $\ex\bigl[D^*(j)\boldsymbol|\,\bold W \bigr]$ for $j\le j_n\boldsymbol{\boldsymbol{:}}=\lfloor n^a\rfloor$, $a\in (0,1)$ to be specified later. By \eqref{C4}-\eqref{C4''}, we have
\begin{equation}\label{|Wi-i|}
\begin{aligned}
&\Bbb P(W_i\le 1.1 i\log n,\,\forall i\ge 1)\ge 1-n^{-1},\\
&\Bbb P(|W_i-i|\le \eps i,\,\forall\, i\ge i(n))\ge 1-\exp(-\Theta(\eps^2i(n))),\quad i(n)\to\infty.
\end{aligned}
\end{equation}
Therefore 
\begin{equation}\label{root+root}
\Bbb P\Bigl(W^{1/2}_{mj}+W^{1/2}_{m(j-1)}\le 2\bigl(1.1mj\log n)^{1/2},\,\forall j\ge 1\Bigr)= 1-O(n^{-1}).
\end{equation}
 Further
\begin{equation}\label{sum=sum1+sum2}
\sum_{i> mj} W_i^{-1/2}=\sum_{i\in (mj,mj_n]}\!\!\!W_i^{-1/2} +\sum_{i\in (mj_n,mn]}\!\!\!W_i^{-1/2}=\boldsymbol{\boldsymbol{:}} \Sigma_1+\Sigma_2.
\end{equation}
By \eqref{|Wi-i|} with $\eps=n^{-\delta}$,  ($\delta<(1-a)/2$), we have
\begin{align}
\Sigma_2&=(1+O(\eps))\!\!\sum_{i\in (mj_n,mn]}\!\!i^{-1/2}=(1+O(\eps))\bigl[2(mn)^{1/2}+O(j_n^{1/2})\bigr]\notag\\
&=\bigl(2+O(\eps+n^{(a-1)/2})\bigr)(mn)^{1/2}=\bigl(2+O(\eps)\bigr)(mn)^{1/2}, \label{Sigma2}
\end{align}
with probability at least $1-O\bigl(\exp(-\Theta(\eps^2 j_n))\bigr)=1-O(n^{-K}),\quad \forall\, K>0$.
Therefore
\begin{equation}\label{>mjn}
\Bbb P\Bigl(\sum_{i>mj_n}W_i^{-1/2}\ge \bigl(2+O(n^{-\delta})\bigr)(mn)^{1/2}\Bigr)= 1-O(n^{-K}).
\end{equation}
Combining \eqref{EDmn(j)=}, \eqref{root+root} and \eqref{>mjn}, and using $j\le j_n=\lfloor n^a\rfloor$, we obtain 
\begin{equation}\label{E[D1n(j)|R]=O}
\Bbb P\Bigl(\mu_j(\bold W)\ge 0.5\bigl(n^{(1-a)/2}\log^{-1/2}n\bigr)\Omega_{mj},\,\forall j\le j_n \Bigr)= 1- O(n^{-1}).  
\end{equation}
Applying the estimate \eqref{C3} to $D^*(j)$, we have: with probability $1-O(n^{-1})$,
\begin{multline*}
\Bbb P\left(\Big|D^*(j)-\mu_j(\bold W)\Big|\ge \eps\mu_j(\bold W)\Big|\,\bold W\right)
\\
\le 2\exp\left(-\Theta\Bigl(\eps^2\mu_j(\bold W)\Bigr)\right)
\le \exp\Bigl(-\Theta\bigl(\eps^2\Omega_{mj} n^{(1-a)/2}\log^{-1/2}n\bigr)\Bigr).
\end{multline*}
Using the union bound, we obtain: with probability $1-O(n^{-1})$,
\begin{equation}\label{CondP<,1}
\begin{aligned}
&\Bbb P\left(\exists\,j\le j_n\boldsymbol{:}\,\Big|D^*(j)-\mu_j(\bold W)\Big|\ge \eps\mu_j(\bold W)
\Big|\,\bold W\right)\\
&\qquad\quad\le \sum_{j\le j_n}\exp\Bigl(-\Theta\bigl(\eps^2\Omega_{mj} n^{(1-a)/2}\log^{-1/2}n\bigr)\Bigr).
\end{aligned}
\end{equation}
Taking the expectations and using 
$
\ex[e^{-\la \Omega_{mj}}]=\Bigl(\ex[e^{-\la w}]\Bigr)^m =(1+\la)^{-m}, 
$
we transform \eqref{CondP<,1} into
\begin{equation}\label{P<}
\begin{aligned}
&\Bbb P\left(\exists\,j\le j_n\boldsymbol{\boldsymbol{:}}\,\Big|D^*(j)-\mu_j(\bold W)\Big|\ge \eps\mu_j(\bold W)
\right)\\
&\le \sum_{j\le j_n}\ex\left[\exp\Bigl(-\Theta\bigl(\eps^2\Omega_{mj} n^{(1-a)/2}\log^{-1/2}n\bigr)\Bigr)\right]+ O(n^{-1})\\
&\le j_n\Bigl[1+\Theta\bigl(\eps^2 n^{(1-a)/2}\log^{-1/2}n\bigr)\Bigr]^{-m}+O(n^{-1})\\
&=O\Bigl(\eps^{-2m}n^{-\frac{m}{2}(1-\frac{a(m+2)}{m})}(\log n)^{m/2}\Bigr) +O(n^{-1}),
\end{aligned}
\end{equation}
and the bound tends to zero if we choose  
\[
a<\frac{m}{m+2},\quad \eps=n^{-\delta}\,\text{ and }\,\delta<\frac{1-\frac{a(m+2)}{m}}{4}. 
\]
Therefore, with high probability, the degrees of the first $j_n=\lfloor n^a\rfloor$ vertices are (uniformly) sharply concentrated around their (conditional) expected values $\mu_j(\bold W)$.

To evaluate sharply $\mu_j(\bold W)$, let us have a look at the sum $\Sigma_1$, which we haven't needed
so far. Now $\max_{j\le j_n}\Sigma_1\le \Sigma_1^*= \sum_{i\in [mj_n]}W_i^{-1/2}$, and
\begin{align*}
\ex\bigl[\Sigma^*_1 \bigr]&=\sum_{i\in [mj_n]}\int_0^{\infty}\frac{z^{i-3/2}}{\Gamma(i)}\,e^{-z}\,dz
=\sum_{i\in [mj_n]}\frac{\Gamma(i-1/2)}{\Gamma(i)}\\
&=\sum_{i\in [mj,mj_n]} O(i^{-1/2}) =O(j_n^{1/2})=O(n^{a/2}).
\end{align*}
Consequently, for $b\in (a/2,1)$, we have $\Bbb P(\Sigma_1\ge n^b) =O(n^{-(b-a/2)})$. Choosing $b<1/2-\delta$, and using \eqref{Sigma2}
we obtain  
\begin{align*}
&\Bbb P\Bigl(\sum_{i>mj}W_i^{-1/2}=(2+O(n^{-\delta}))(mn)^{1/2},\,\forall j\le j_n\Bigr)\ge 1-O(n^{-\Delta}),\\
&\qquad\quad\forall\,\Delta<\min\Bigl(\frac{1-a}{2}-\delta;\,\frac{m}{2}\Bigl(1-\frac{a(m+2)}{m}\Bigr)-2m\delta\Bigr).
\end{align*}
So, by \eqref{EDmn(j)=}, 
\begin{multline}\label{P(EDjappr)}
\Bbb P\Bigl(\mu_j(\bold W)=\!(2+O(\eps))(mn)^{1/2}\\
\times \bigl(W^{1/2}_{mj}-W^{1/2}_{m(j-1)}\bigr),\,\forall j\le j_n\Bigr)= 1-O(n^{-\Delta}).
\end{multline}
{\bf (2)\/}  It remains to get a likely lower bound for $\min\limits_{j\le j_n}n^{1/2}(W_{mj}^{1/2}-W_{m(j-1)}^{1/2})$. We know that
\[
W^{1/2}_{mj}-W^{1/2}_{m(j-1)}=\frac{\Omega_{mj}}{W^{1/2}_{mj}+W^{1/2}_{m(j-1)}}\ge \frac{\Omega_{mj}}{2W^{1/2}_{mj}}.
\]
Therefore
\begin{align*}
&\qquad\qquad\qquad\qquad\Bbb P\left(n^{1/2}\min_{j\in [j_n]}\Bigl(W^{1/2}_{mj}-W^{1/2}_{m(j-1)}\Bigr)\le n^b\right)\\
&\le \Bbb P\bigl(\max_{j\in [j_n]}(mj)^{-1}W_{mj}\ge  1.1\log n\bigr)+\Bbb P\bigl(\min_{j\in [j_n]} j^{-1/2}\,\Omega_{mj}\le 3n^{b-1/2}\sqrt{m\log n}\Bigr)\\
&\qquad\qquad\qquad\qquad\qquad\qquad =\boldsymbol{:} \Bbb P_1+\Bbb P_2.
\end{align*}
Here $\Bbb P_1\le n^{-1}$.  Further, choosing $b<(1-a)/2$, and denoting $\sigma_n=3n^{b-(1-a)/2}\sqrt{m\log n}$,  we obtain
\begin{align*}
\Bbb P_2&\le\sum_{j\in [j_n]}\Bbb P\bigl(\Omega_{mj}\le 3j^{1/2}n^{b-1/2}\sqrt{m\log n}\bigr)\\
&\le j_n\!\int\limits_{x\le\sigma_n}\!\!\!e^{-x}\,\frac{x^{m-1}}{\Gamma (m)}\,dx
=O\left( n^a \Bigl(n^{b-(1-a)/2}(\log n)^{1/2}\Bigr)^m\right)\\
&=O\left(n^{-m\bigl[\bigl(1-\tfrac{a(m+2)}{m}\bigr)/2-b\bigr]} (\log n)^{m/2}\right)\to 0,
\end{align*}
if $b<\bigl(1-\tfrac{a(m+2)}{m}\bigr)/2$. Therefore
\begin{align*}
&\Bbb P\left(n^{1/2}\min_{j\in [j_n]}\Bigl(W^{1/2}_{mj}-W^{1/2}_{m(j-1)}\Bigr)\ge n^b\right)= 1-O(n^{-\tilde\Delta}),\\
&\forall\,b<\bigl(1-\tfrac{a(m+2)}{m}\bigr)/2\,\text{ and }\,\tilde\Delta<m\bigl[\bigl(1-\tfrac{a(m+2)}{m}\bigr)/2-b\bigr].
\end{align*}
\end{proof}
\begin{corollary}\label{rate} Suppose $a$, $\delta$ and $\Delta$ meet the conditions of Theorem \ref{degs,m}.
Denote $\bold D=\{D(j)\boldsymbol{:} j\in [j_n]\}$, $(j_n=\lfloor n^a\rfloor)$, $\boldsymbol {\mathcal D}=\{\mathcal D(j)\boldsymbol{:} j\in [n^a]\}$,
and 
\[
\|\bold D-\boldsymbol {\mathcal D}\|_{\ell_1}=\sum_{j\in [n^a]}\boldsymbol|D(j)-\mathcal D(j)\boldsymbol|.
\]
There is a constant $c=c(a,\delta)>0$ such that
\begin{align*}
&\qquad\,\,\Bbb P\Bigl(n^{-1/2}\|\bold D-\boldsymbol {\mathcal D}\|_{\ell_1} \le c n^{-\delta}\Bigr)\ge 1- n^{-\Delta}.
\end{align*}
\end{corollary}
\begin{proof} Immediate, since
\[
\sum_{j\in [j_n]}\bigl(W_{mj}^{1/2}-W_{m(j-1)}^{1/2}\bigr)=W_{mj_n}^{1/2}\le (2mj_n)^{1/2},
\]
with probability $1-\exp(-\Theta(j_n))$.
\end{proof} 

\subsection{Likely degree bounds for the next $n(1-\sigma)-j_n$ vertices.}
\begin{theorem}\label{second} Given $\sigma\in (0,1)$, there exists a unique root  $z(\sigma)\in (1,\infty)$ of the equation
\[
(1-\sigma)^{-1/2}-1=\varphi(z),\quad \varphi(z):=(z\log z+1-z)^{-1};
\]
$z(\sigma)\sim 2 (\sigma\log (1/\sigma))^{-1}$ as $\sigma\to 0$.  If $z>z(\sigma)$ then
\[
\lim_{n\to\infty}\Bbb P\Bigl(D(j)\le z \Bigl((n/j)^{1/2}-1\Bigr)\log n,\quad\forall j\in [j_n, (1-\sigma)n]\Bigr)=1.
\]
\end{theorem}

{\bf Note.\/} So while the vertices close to the top of the roster are (whp) of degrees of order $\Theta(n^1/2)$, the vertices filling, say,
the interval $[\eps n, (1-\sigma)n]$ are of degrees $O(\log n)$.

\begin{proof} A minor variation of the argument involving \eqref{|Wi-i|}-\eqref{E[D1n(j)|R]=O}, and using
$
\sum_{i>mj}i^{-1/2}\le 2\bigl(\sqrt{mn}-\sqrt{mj}\bigr),
$
shows that, for $\delta<a/2$,
\[
\Bbb P\Biggl(\!\ex\bigl[D^*(j)\boldsymbol|\,\bold W \bigr]\!\le (1+n^{-\delta})\,\Omega_{mj}\Bigl((n/j)^{1/2}-1\Bigr),\,\forall\,j\in [j_n,n]\!\Biggr)\!
\ge 1-n^{-K}.
\]
Further, the $n$ variables $\Omega_{mx}$ are i.i.d. random variables, and with $y_n:=\log n+m\log\log n$ we bound
\begin{align*}
\Bbb P(\Omega_m\ge y_n)=&\int_{y\ge y_n}\frac{y^{m-1} e^{-y}}{\Gamma(m)}\,dy
\le \frac{y_n^{m-1} e^{-y_n}}{\Gamma(m)}=O(n^{-1}\log^{-1}n),
\end{align*}
proving that 
\[
\Bbb P\bigl(\max_{x\in [n]}\Omega_{mx}\ge \log n+m\log\log n\bigr)=O(\log^{-1}n).
\]
Therefore whp for all $j\in [j_n,n]$
\[
\ex\bigl[D^*(j)\boldsymbol|\,\bold W \bigr]\!\le E_n(j):=(\log n+(m+1)\log\log n)\Bigl((n/j)^{1/2}-1\Bigr).
\]
By \eqref{X>mu},  for every $z>1$, whp
\[
\Bbb P(D^*(j)\ge zE_n(j)\boldsymbol|\,\bold W)\le\exp\bigl(-E_n(j)\varphi(z)\bigr).
\]
For $z>z(\sigma)$, we have $e(n,z):=\min\{E_n(j)\phi(z)/\log n: j\le (1-\sigma)n\}>1$, and $e(n,z)$ is bounded away from $1$ as $n\to\infty$. Therefore whp
\begin{align*}
&\Bbb P\Bigl(\exists j\in [j_n, (1-\sigma) n]\boldsymbol{:} D^*(j)\ge z  \Bigl((n/j)^{1/2}-1\Bigr)\log n\boldsymbol|\,\bold W\Bigr)\\
&\le\sum_{ j=j_n}^{(1-\sigma) n}\Bbb P\Bigl(D^*(j)\ge z  \Bigl((n/j)^{1/2}-1\Bigr)\log n\boldsymbol|\,\bold W\Bigr) =O\bigl(n^{-e(n,z)+1}\bigr)
\to 0.
\end{align*}
Taking expectation with respect to $\bold W$ we complete the proof. 
\end{proof}

\subsection{Loops and multiple edges}  Each loop at a vertex $v$ contributes $2$ to the degree of $v$, and 
each pair $(\ell_j,R_j)$ contributes $1$ to the degrees of the vertices containing $\ell_j$ and $R_j$. To get a simple graph we need to discard
the loops and to identify the parallel edges. How substantial is the attendant decrease of the vertex degrees? \\

{\bf (1)\/}  Let us begin with loops.
Vertex $1$ contributes the maximum number $m$ of loops.  Consider vertex $j>1$. It contains
$m$ right endpoints $R_{m(j-1)+1},\dots R_{mj}$. The chord $(\ell_{m(j-1)+t}, R_{m(j-1)+t})$ forms a loop at  vertex $j$ if and only if
$\ell_{m(j-1)+t}$ belongs to $j$-th vertex, meaning that $\ell_{m(j-1)+t}\in (R_{m(j-1)}, R_{m(j-1)+t})$. Therefore, denoting $L_n$ the total number of loops contributed by all vertices $j\ge1$, we have
\[
L_n=m+\sum_{j>1}\sum_{t=1}^m\Bbb I(R_{m(j-1)}<\ell_{m(j-1)+t}<R_{m(j-1)+t}).
\]
There are $m(n-1)$ event indicators in this sum; conditioned on $\bold W$ they are all independent. We plan to evaluate
sharply $\ex[L_n\boldsymbol|\,\bold W]$ and to show that $L_n$ is concentrated around $\ex[L_n\boldsymbol|\,\bold W]$.
To begin,
\begin{multline}\label{denom'}
\ex[L_n\boldsymbol|\,\bold W]-m=\sum_{j>j1}\sum_{t=1}^m \int_{R_{m(j-1)}}^{R_{m(j-1)+t}}\frac{dx}{2\sqrt{xR_{m(j-1)+t}}}\\
=\sum_{j>1}\sum_{t=1}^m \frac{R_{m(j-1)+t}^{1/2}-R_{m(j-1)}^{1/2}}{R_{m(j-1)+t}^{1/2}}
= \sum_{j>1}\sum_{t=1}^m \frac{W_{m(j-1)+t}^{1/2}-W_{m(j-1)}^{1/2}}{W_{m(j-1)+t}^{1/2}}\\
=\sum_{j>1}\sum_{t=1}^m \frac{W_{m(j-1)+t}-W_{m(j-1)}}{\bigl(W_{m(j-1)+t}^{1/2}+W_{m(j-1)}^{1/2}\bigr) W_{m(j-1)+t}^{1/2}}.
\end{multline}
We estimate
\begin{align*}
&\frac{1}{2W_{m(j-1)}}-\frac{1}{\bigl(W_{m(j-1)+t}^{1/2}+W_{m(j-1)}^{1/2}\bigr) W_{m(j-1)+t}^{1/2}}\\
&=\frac{\Bigl(W_{m(j-1)+t}-W_{m(j-1)}\Bigr)+W^{1/2}_{m(j-1)}\Bigl(W^{1/2}_{m(j-1)+t}-W^{1/2}_{m(j-1)}\Bigr)}{2\bigl(W_{m(j-1)+t}^{1/2}+W_{m(j-1)}^{1/2}\bigr) W_{m(j-1)} W_{m(j-1)+t}^{1/2}}.
\end{align*}
The denominator is $4W_{m(j-1)}^{2}$, at least. The numerator equals
\[
\Bigl(\!W_{m(j-1)+t}-W_{m(j-1)}\!\Bigr)\!\!\left(\!\!1+\frac{W^{1/2}_{m(j-1)}}{W^{1/2}_{m(j-1)+t}+W^{1/2}_{m(j-1)}}\right)\!\!\le 
2\Bigl(\!W_{m(j-1)+t}-W_{m(j-1)}\!\Bigr).
\]
Therefore replacing the denominator in \eqref{denom'} with $2W_{m(j-1)}$, independent of $t$, results in additive error of the order
\[
X_{n}\boldsymbol{\boldsymbol{:}}=\sum_{j>1}\sum_{t=1}^m \frac{\bigl(W_{m(j-1)+t}-W_{m(j-1)}\bigr)^2}{W^2_{m(j-1)}}.
\]
Since $W_{m(j-1)+t}-W_{m(j-1)}$ and $W^2_{m(j-1)}$, are independent and $W_{m(j-1)+t}-W_{m(j-1)}\overset{\mathcal D}\equiv
\sum_{a\in [m]}w_a$, the expected value of the generic fraction in the double sum equals
\[
\ex\Biggl[\Bigl(\sum_{a\in [m]}w_a\Bigr)^2\Biggr]\cdot \ex\Bigl[W^{-2}_{m(j-1)}\Bigr] =O\Bigl(\frac{\Gamma(m(j-1)-3)}{\Gamma(m(j-1)-1)}\Bigr)=O(j^{-2}).
\]
Therefore
$
\ex[X_{n}]=O\Bigl(\sum_{j\ge 1}j^{-2}\Bigr)=O(1).
$
That is, with the $t$-independent denominator $2W_{m(j-1)}$ in place, the additive error is bounded in expectation. Furthermore
\[
\sum_{t=1}^m\bigl(W_{m(j-1)+t}-W_{m(j-1)}\bigr) =\sum_{a=m(j-1)+1}^{mj}[m(j-1)-a] w_a.
\]
So it remains to evaluate
\[
E^{(1)}_{n}\boldsymbol{\boldsymbol{:}}=\frac{1}{2}\sum_{j>1} W_{m(j-1)}^{-1}\,Y_j, \quad Y_j\boldsymbol{\boldsymbol{:}}=\!\!\!\!\sum_{a=m(j-1)+1}^{mj}\!\!\!\!\! [m(j-1)-a]\, w_a.
\]
$Y_j$ are i.i.d. random variables with $\ex[Y_j]=m(m+1)/2$, and $Y_j$ is independent of $W_{m(j-1)}$. Introduce $E^{(2)}_{n}=\frac{1}{2}\sum_{j>1}(m(j-1))^{-1}Y_j$.
Then, by Cauchy-Schwartz inequality,
\begin{align*}
\ex\bigl[|E^{(1)}_{n}-E^{(2)}_{n}|\bigr]&\le \frac{1}{2}\sum_{j>1}\ex^{1/2}\bigl[(W^{-1}_{m(j-1)}-(m(j-1))^{-1})^2\bigr]\cdot \ex^{1/2}[Y_j^2]\\
&=\frac{\ex^{1/2}[Y_1^2]}{2}\sum_{j>1}\left\{\left[\frac{\Gamma(\nu-2)}{\Gamma(\nu)}-\frac{2}{\nu}\frac{\Gamma(\nu-1)}{\Gamma(\nu)}+\frac{1}{\nu^2}
\right]_{\nu=m(j-1)}\right\}^{1/2}\\
&=O\left(\sum_{\ell>1}\ell^{-3/2}\right)=O(1).
\end{align*}
Therefore $|E_{n}^{(1)}-E_{n}^{(2)}|$ is also bounded in expectation. Finally, 
\[
\ex[E_{n}^{(2)}]=\frac{m(m+1)/2}{2m}\sum_{j=2}^n(j-1)^{-1}=\frac{m+1}{4}\log n +O(1),
\]
and it is easy to see that 
$
\ex\Bigl[\bigr(E^{(2)}_{n}\bigr)^2] =\Bigl(\ex[E_{n}^{(2)}]\Bigr)^2 +O(1).
$
So the variance of $E^{(2)}_{n}$ is bounded, as well. Collecting the pieces we end up with
\begin{lemma}\label{exLnm=} Let $L_{n}$ stand for the total number of loops in $G_m^n$. Let $\ex[L_{n}|\,\bold W]$ denote the conditional expected number of loops. Then
\[
\ex[L_{n}|\,\bold W]=\frac{m+1}{4}\log n+\mathcal L_{n,m},
\]
where $\ex[|\mathcal L_{n}|]<\gamma$ for a constant $\gamma=\gamma(m)$.
\end{lemma}
\begin{theorem}\label{Ln} For every $\eps\in (0,3/2)$ and $\delta\in (0,1)$,
\[
\Bbb P\left(\left|L_{n}-\ex[L_{n}|\,\bold W]\right|\ge \eps \ex[L_{n}|\,\bold W]\right)=O\bigl(\log^{-(1-\delta)}n\bigr).
\]
\end{theorem}
\begin{proof} By Lemma \ref{exLnm=}, 
\[
\Bbb P\left(\left|\ex[L_{n}\boldsymbol|\,\bold W]-\frac{m+1}{4}\log n\right|>\log^{1-\delta}n\right)=O(\log^{-(1-\delta)}n).
\]
So invoking \eqref{C3}, with probability $\ge 1-O(\log^{-(1-\delta)}n)$ we have
\begin{align*}
&\quad\Bbb P\left(\left|L_{n}-\ex[L_{n}|\,\bold W]\right|\ge \eps \ex[L_{n,m}|\,\bold W]\boldsymbol|\,\bold W\right)\\
&\le \exp\Bigl(-\Theta\bigl(\eps^2 \ex[L_{n,m}|\,\bold W]\bigr)\Bigr)=\exp\Bigl(-\Theta(\eps^2\log n)\Bigr).
\end{align*}
Taking expectations we complete the proof.
\end{proof}

{\bf (2)\/} Turn to parallel edges.
\begin{theorem}\label{mathcalP<} Let $\mathcal P_{n}$ stand for the total number of {\bf pairs\/} of parallel edges in $G_m^n$.
Whp $\mathcal P_{n}$  is asymptotic to $\frac{m^2-1}{16}\log^2n$.
Thus whp the identification operation  reduces the edge count by $\Theta(\log^2n)$.
\end{theorem}

\begin{proof} First of all,
\[
\mathcal P_{n}=\!\!\sum_{1\le a<b\le n}\!\!\!\mathcal P_n(a,b),\quad \,\,\mathcal P_n(a,b):=\!\!\!\sum_{m(b-1)< i<j\le mb}\!\!\!\!\!\!\!\!\!\Bbb I\Bigl(R_{m(a-1)}\le \ell_i,\,\ell_j\le R_{ma}\Bigr),
\]
$(R_0\boldsymbol{:}=0)$. Here $\mathcal P_n(a,b)$ is the total number of pairs of parallel edges connecting the vertices $a$ and $b$.
Indeed, the generic indicator in the sum is $1$ if there are two right endpoints $R_i$ and $R_j$ in vertex $b$ whose left partners $\ell_i$
and $\ell_j$ are situated between the last right endpoint in vertex $a-1$ and the last right endpoint in vertex $a$. 

Conditioned upon $\bold W$, we have a ``balls and bins'' allocation scheme, with the left endpoints playing the role of balls and the set 
of intervals $(R_{u-1},R_u]$ playing the role of bins. 
The left end partner $\ell_i$ of $R_i$ selects, independently of all other left endpoints,  the interval $(R_{u-1},R_u]$ for $u\le i$ with conditional probability $\Bbb P\bigl(R_i,R_u\bigr)=\frac{R^{1/2}_u-R^{1/2}_{u-1}}{R^{1/2}_u}$,
so that $\sum_{u\le i}\Bbb P\bigl(R_i,R_u\bigr)=1$, 
Introduce
$\Bbb I\bigl(R_i,R_u\bigr)$ the indicator of the event ``ball $\ell_i$ selected bin $(R_{u-1},R_u]$''. These indicators are (conditionally) independent for the distinct $i$s
and negatively associated for the same $i$. Furthermore, 
\[
\mathcal P_n(a,b)=\!\!\!\sum_{m(b-1)< i<j\le mb}\left(\sum_{u=m(a-1)+1}^{ma}\!\!\!\!\!\Bbb I(R_i, R_u)\right)\cdot \left(\sum_{v= m(a-1)+1}^{ma}\!\!\!\!\!\Bbb I(R_j,R_v)\right).
\]
So (1) each $\mathcal P_n(a,b)$ is a non-decreasing function of the indicators on the RHS, and (2) the two groups of indicators for $(a,b)$
and $(a',b')\neq (a,b)$ are disjoint. By a general theorem, 
(Dubhashi and Ranjan \cite{DubRan}, (Proposition $8$), Joag-Dev and  Proschan \cite{JoaPro}), the $\mathcal P_n(a,b)$ 
are negatively associated as well. Likewise the indicators $\Bbb I((a,b)\in E(G_m^n))$ are negatively associated as well.
(Similarly each $\Bbb I((a,b)\notin E(G_m^n))$ is a {\it decreasing\/} function of the corresponding indicators $X\bigl(R_i,R_j)$; by the same theorem the indicators $\Bbb I((a,b)\notin E(G_m^n))$ are negatively associated too.) 

Besides $\mathcal P_n(a,b)\le m^3$. Consequently,
\[
\ex[\mathcal P_n(a,b) \mathcal P_n(a',b')\boldsymbol|\,\bold W]\le\!\left\{\begin{aligned}
&\ex[\mathcal P_n(a,b)\boldsymbol|\,\bold W]\,\ex[\mathcal P_n(a',b')\boldsymbol|\,\bold W],\!\!&(a,b)\neq (a',b'),\\
&m^3\ex[\mathcal P_n(a,b)],\!\!& (a,b)=(a',b').\end{aligned}\right.
\]
It follows that
\[
\ex[\mathcal P_n^2\boldsymbol|\,\bold W]\le \ex^2[\mathcal P_n\boldsymbol|\,\bold W]+m^3\ex[\mathcal P_n\boldsymbol|\,\bold W]
\Longrightarrow \frac{\text{Var}(\mathcal P_n\boldsymbol|\,\bold W)}{E^2[\mathcal P_n\boldsymbol|\,\bold W]}\le\frac{m^3}{\ex[\mathcal P_n\boldsymbol|\,\bold W]},
\]
so that
\begin{equation}\label{Pnconc}
\Bbb P\Bigl(|\mathcal P_n-\ex[\mathcal P_n\boldsymbol|\,\bold W]|\ge \eps\,\ex[\mathcal P_n\boldsymbol|\,\bold W]\Bigr)\le 
\frac{m^3}{\eps^2\ex[\mathcal P_n\boldsymbol|\,\bold W]}.
\end{equation}
Now $\ex[\mathcal P_n\boldsymbol|\,\bold W]=\sum_{1\le a<b\le n}\ex[\mathcal P_n(a,b)\boldsymbol|\,\bold W]$,
and
\begin{align*}
&\ex[\mathcal P_n(a,b)\boldsymbol|\,\bold W]=\!\!\!\sum_{m(b-1)< i<j\le mb}\!\!\left(\sum_{u=m(a-1)+1}^{ma}\!\!\!\!\!\!\!\Bbb P(R_i, R_u)\!\!\right) \!\!\left(\sum_{v= m(a-1)+1}^{ma}\!\!\!\!\!\!\!\Bbb P(R_j,R_v)\!\!\right)\\
&=\sum_{m(b-1)< i<j\le mb}\!\!\!\frac{\left(R_{ma}^{1/2}-R_{m(a-1)}^{1/2}\right)^2}{R_i^{1/2}R_j^{1/2}}=
\sum_{m(b-1)< i<j\le mb}\!\!\!\frac{\left(W_{ma}^{1/2}-W_{m(a-1)}^{1/2}\right)^2}{W_i^{1/2}W_j^{1/2}}.
\end{align*}
{\bf (1)\/} So
\[
\Sigma_1:=\sum_{b=2}^ n\ex[\mathcal P_n(1,b)\boldsymbol|\,\bold W]\le m^2\sum_{b=2}^n\frac{W_m}{W_{m(b-1)+1}},
\]
and, using
\[
\ex[W_{\mu}^{\sigma}]=\frac{\Gamma(\mu+\sigma)}{\Gamma(\mu)}=O(\mu^{\sigma}),\quad \sigma>-\mu,
\]
we bound
\begin{align*}
\ex[\Sigma_1] &\le m^2\sum_{b=2}^n \ex^{1/2}[W_m^2]\cdot\ex^{1/2}[W_{m(b-1)+1}^{-2}]\\
&=O\left(\sum_{2\le b\le n}(b-1)^{-1}\right)=O(\log n).
\end{align*}
Next
\[
\Sigma_2:=\sum_{2\le a\le \log n,\atop a< b\le n}\ex[\mathcal P_n(a,b)\boldsymbol|\,\bold W]\le m^2\sum_{2\le a\le \log n\atop a<b\le n}\frac{\Omega_{ma}^2}{W_{ma}W_{m(b-1)+1}},
\]
so, by H\"older inequality,
\begin{align*}
\ex[\Sigma_2] &\le m^2\sum_{2\le a\le \log n\atop a<b\le n}\ex^{1/3}\bigl[\Omega^6_{ma}\bigr]\ex^{1/3}\bigl[W_{ma}^{-3}\bigr]
\ex^{1/3}\bigl[W_{m(b-1)+1}^{-3}\bigr]\\
&=O\left(\sum_{2\le a\le \log n\atop a<b\le n}\frac{1}{a (b-1)}\right)=O(\log n\cdot\log\log n).
\end{align*}
Therefore 
$
\ex\bigl[\Sigma_1+\Sigma_2\bigr]=O(\log n\cdot\log\log n),
$
or $\Sigma_1+\Sigma_2=O_p(\log n\cdot\log\log n)$, i. e. whp $\Sigma_1+\Sigma_2$ scaled by $\log n\cdot\log\log n$ is bounded as
$n\to\infty$. 

{\bf (2)\/} Turn to the remaining part of $\ex[\mathcal P_n\boldsymbol|\,\bold W]$, namely $\Sigma^*:=\sum_{\log n\le a < b\le n}$\linebreak $\ex[\mathcal P_n(a,b)\boldsymbol|\,\bold W]$. By \eqref{C4''}, with probability
$1-\exp\Bigl(-\Theta(\log^{1-2\sigma}n)\bigr)$, each of the $W_{\nu}$ involved is within the factor $1+O(\log^{-\sigma}n)$ from $\nu=\ex[W_{\nu}]$ . So with probability that high,
\[
\Sigma^*=\bigl(1+O(\log^{-\sigma}n)\bigr)\binom{m}{2}m^{-2}\sum_{\log n\le a < b\le n}\frac{\Omega_{ma}^2}{4ab}.
\]
The $\Omega_{ma}$ are i.i.d. variables with $\ex[\Omega_{ma}^2]=(m+1)m$. So the expected value of the sum is
asymptotic to
\[
\binom{m}{2}m^{-2}\frac{(m+1)m}{4}\sum_{1\le a<b\le n}\frac{1}{ab}\sim \frac{m^2-1}{16} \log^2 n,
\]
while $\ex\left[\sum_{1\le a<b\le n}\frac{\Omega_{ma}^{4}}{a^2b^2}\right]=O(1)$, i.e. the sum of the squared terms is bounded in probability. It follows
that whp $\Sigma^*$ is sharply concentrated around $\frac{m^2-1}{16} \log^2 n$. But then so is the whole $\ex[\mathcal P_n\boldsymbol|\,\bold W]$, since $\Sigma_1+\Sigma_2=O_p(\log n\cdot\log\log n)$. In combination with \eqref{Pnconc} this completes the
proof.
\end{proof}

\section{Two vertex sets without any connecting edges} \label{basic} Let us consider a basic problem: bound the probability $\Bbb P(A,B)$ that, for two disjoint sets of vertices $A\subset [n]$ and $B\subset [n]$, ($|A|=\mu$, $|B|=\nu$), there is no edge $(a,b)$ with $a\in A$ and $b\in B$, i.e. formally
\[
\Bbb P(A,B)= \Bbb P\!\left(\bigcap_{a\in A,\,b\in B}\left\{(a,b)\notin E(G_m^n)\right\}\right).
\]
We focus on the pairs $(A,B)$ such that $\mu+\nu= (1-\delta)n$, $\delta=\delta(n)\in (0,1)$, being bounded away from $0$ and $1$.

Begin with $P(A,B\boldsymbol|\bold W)$, the probability of the above event conditioned on $\bold W=\{W_i\}_{i\in [mn+1]}$.
For $\a<\be$, let $\a\leftarrow \be$ denote the event ``one of the right endpoints in $\be$ has its left endpoint in $\a$''. Then
\[
\bigcap_{a,\,b}\left\{(a,b)\notin E(G_m^n)\right\}=\bigcap_{a<b}\{a\not\leftarrow b\}\cap\bigcap_{a>b}\{b\not\leftarrow a\},\quad a\in A,\,b\in B,
\]
and conditioned on $\bold W$, two groups of events, $\{a\not\leftarrow b\}$, $(a\in A,b\in B)$, and $\{b\not\leftarrow a\}$, $(a\in A, b\in B)$,   are independent of each other. Conditioned on $\bold W$, within each group the events are negatively associated. (See the proof of Lemma 
\ref{mathcalP<}.)

For $a<b$, the event $a\not\leftarrow b$ means that none
of the right endpoints $R_{m(b-1)+1},\dots, R_{mb}$ have their left endpoints between $R_{m(a-1)}$ and $R_{ma}$. Therefore
\begin{align}
&\qquad\qquad\Bbb P(a\not\leftarrow b\boldsymbol|\,\bold W)=\prod_{i=m(b-1)+1}^{mb}\left(1-\int_{R_{m(a-1)}}^{R_{ma}}\frac{dx}{2(R_i x)^{1/2}}\right)\notag\\
&\le\exp\left(-\sum_{i=m(b-1)+1}^{mb}\frac{R^{1/2}_{ma}-R^{1/2}_{m(a-1)}}{R_i^{1/2}}\right)\le\exp\left(-m\,\frac{W^{1/2}_{ma}-W^{1/2}_{m(a-1)}}{W_{mb}^{1/2}}\right)\notag\\
&\qquad\qquad\qquad\qquad\le\exp\left(-\frac{m\, \Omega_{ma}}{2(W_{ma}W_{mb})^{1/2}}\right).\label{basic}
\end{align}
Using the conditional independence/negative association of edge indicators, we multiply the bounds \eqref{basic} and their counterparts for $a>b$ over all pairs $(a,b)$, ($a\in A, b\in B$), and obtain:
\begin{equation}\label{P(A,B|W)<}
\Bbb P(A,B\boldsymbol|\,\bold W)\le \exp\!\left(\!-\sum_{a < b}\frac{m\,\Omega_{ma}}{2\sqrt{W_{ma}W_{mb}}}-\sum_{a>b}\frac{m\,\Omega_{mb}}{2\sqrt{W_{ma}W_{mb}}}\right).
\end{equation}
A direct evaluation of the expected RHS expression is out of question. We know and already used the facts  that the sum $W_{mj}$ increases with $j$, and that for 
$j$ large enough $W_{mj}$ is sharply concentrated around its expected value $mj$. ``Freezing'' $\Omega_{ma}$, $\Omega_{mb}$, let us push
the elements of $C:=A\cup B$ all the way to the right, {\it preserving the initial ordering\/} of the elements of $A$ and $B$. Let $\mathcal A$, $\mathcal B$ denote
the terminal ``destinations'' of $A$ and $B$, and $\mathcal C=\mathcal A\cup \mathcal B=[n-\mu-\nu+1,n]$. Then $\min(\mathcal C)=n-\mu-\nu+1\ge \delta n$, 
implying that  
\begin{equation}\label{mathcal W=}
\Bbb P(\mathcal W^c)<e^{-\Theta(\log^2n)},\quad \mathcal W\boldsymbol{:}=\bigcup_{j=n(1-\delta )}^n\!\!
\left\{\frac{W_{mj}}{mj}\le 1+n^{-1/2}\log n\right\}.
\end{equation}
By the definition of $\mathcal W$,
\begin{equation}\label{I(W)P<}
\Bbb I(\mathcal W)\cdot \Bbb P(A,B\boldsymbol|\,\bold W)\le \exp\!\left(\!-c\sum_{a < b}\frac{\Omega_{ma}}{\sqrt{a_tb_t}}-c\sum_{a>b}\frac{\Omega_{mb}}{\sqrt{a_tb_t}}\right);
\end{equation}
here $c=0.5-o(1)$, and $a_t\in \mathcal A$ and $b_t\in \mathcal B$ are the terminal destinations of $a\in A$ and $b\in B$
respectively. Let $\Bbb P(A,B; \mathcal W)$ denote the probability that there is
no edge between $A$ and $B$ and the event $\mathcal W$ holds. Then we have
\begin{equation}\label{P(A,B:W)<E[exp(-)]}
\Bbb P(A,B; \mathcal W)\le \ex\bigl[e^{-cS}\bigr],\quad S:=\sum_{a < b}\frac{\Omega_{ma}}{\sqrt{a_tb_t}}+\sum_{a>b}\frac{\Omega_{mb}}{\sqrt{a_tb_t}}.
\end{equation}
\subsection{Concentration of $S$ around $\ex[S]$.} The $\mu+\nu$ random variables,  $\Omega_{ma}$, $(a\in A)$, and $\Omega_{mb}$, $(b\in B)$, are independent, each being distributed
as $\Omega_m\boldsymbol{:}=\sum_{t=1}^m w_t$. Let us show that $S$ is sharply concentrated around $\ex[S]$. By the definition of $S$, we have
\[
S=\sum_kd_kw_k,\quad d_k=\left\{\begin{aligned}
&\!\!\sum_{b\in B\boldsymbol{:}\, b>a}(a_tb_t)^{-1/2},&&k\in [m(a-1)+1,ma],\\
&\!\!\sum_{a\in A\boldsymbol{:}\,a>b}(a_tb_t)^{-1/2},&&k\in [m(b-1)+1,mb],\end{aligned}\right.
\]
where the $m(\mu+\nu)$ variables $w_k$ are independent exponentials.
Then, by \eqref{C5}, we have 
\begin{equation}\label{S(AW)conc}
\Bbb P\Bigl(S\le (1-\eps)\ex[S]\Bigr)\le\exp\Biggl(-\frac{\eps^2}{2}\frac{\Bigl(\sum_k d_k\Bigr)^2}{\sum_kd_k^2}
\Biggr),\quad \eps\in(0,1).
\end{equation}
Here
\begin{equation}
\begin{aligned}
&\qquad\qquad\qquad\quad\quad\quad \ex[S]=\sum_k d_k\label{ESAW>}\\
&\qquad=m\sum_{a_t}\sum_{b_t>a_t} (a_tb_t)^{-1/2}+m\sum_{b_t}\sum_{a_t>b_t} (a_tb_t)^{-1/2}\\
&\quad\quad=m\!\!\sum_{a_t,\,b_t}(a_tb_t)^{-1/2}=m\Biggl(\sum_{a_t}a_t^{-1/2}\!\Biggr)\Biggl(\sum_{b_t}b_t^{-1/2}\!\Biggr).
\end{aligned}
\end{equation}
Next
\begin{align*}
&\sum_k d_k^2=m\sum_{a_t}\Biggl(\sum_{b_t>a_t}(a_tb_t)^{-1/2}\!\Biggr)^2
+m\sum_{b_t}\Biggl(\sum_{a_t>b_t}(a_tb_t)^{-1/2}\!\Biggr)^2\\
& \le m\Biggl(\sum_{a_t}a_t^{-1}\Biggr)\Biggl(\sum_{b_t}b_t^{-1/2}\Biggr)^2
+ m\Biggl(\sum_{b_t}b_t^{-1}\Biggr)\Biggl(\sum_{a_t}a_t^{-1/2}\Biggr)^2.
\end{align*}
Therefore
\begin{align*}
\frac{\Bigl(\sum_k d_k\Bigr)^2}{\sum_kd_k^2}&\ge m \Biggl(\frac{\sum_{a_t}a_t^{-1}}{\left(\sum_{a_t}a_t^{-1/2}\right)^2}+
\frac{\sum_{b_t}b_t^{-1}}{\left(\sum_{b_t}b_t^{-1/2}\right)^2}\Biggr)^{-1}.
\end{align*}
Observe that $\min(\mathcal A),\,\min(\mathcal B)\ge \delta n$, and $\max([n])=n$. A classic Kantorovich-Schweitzer inequality (\cite{Kan}, \cite{Sch})
states: if $0<x\le x_i\le X$, $\xi_i\ge 0$, $\sum_i \xi_i=1$, then 
\[
\Bigl(\sum_i \xi_i x_i\Bigr)\cdot\Bigl(\sum_i \xi_i x_i^{-1}\Bigr)\le \frac{(X+x)^2}{4Xx}.
\]
Therefore, as $|\mathcal A|=\mu$, we bound
\begin{align}
\frac{\sum_{a_t}a_t^{-1}}{\left(\sum_{a_t}a_t^{-1/2}\right)^2}&\le \mu^{-1}\Bigl(\sum_{a_t}\frac{a_t^{-1/2}}{\sum_{\hat a_t} \hat a_t^{-1/2}}\cdot a_t^{-1/2}\Bigr)\Bigl(\sum_{a_t}\frac{a_t^{-1/2}}{\sum_{\hat a_t} \hat a^{-1/2}}\cdot a_t^{1/2}\Bigr)\notag\\
&\qquad\qquad\le \mu^{-1}\frac{(1+\sqrt{\delta})^2}{4\sqrt{\delta}},\label{Kant}
\end{align}
and likewise
\[
\frac{\sum_{b_t}b_t^{-1}}{\left(\sum_{b_t}b_t^{-1/2}\right)^2}\le \nu^{-1}\frac{(1+\sqrt{\delta})^2}{4\sqrt{\delta}}.
\]
Combining the last two bounds and \eqref{S(AW)conc} we arrive at
\begin{lemma}\label{SapprE[S]} If $\mu+\nu=(1-\delta)n$, then for every $\eps\in (0,1)$, 
\begin{equation}\label{P(Ssmall)<}
\begin{aligned}
\Bbb P\Bigl(S\le (1-\eps)\ex[S]\Bigr)&\le\exp\left(-m\,c(\eps,\delta)\frac{\mu\nu}{\mu+\nu}\right),\\
c(\eps,\delta)&:=\frac{2\eps^2\sqrt{\delta}}{(1+\sqrt{\delta})^2}.
\end{aligned}
\end{equation}
\end{lemma}
\begin{corollary}\label{break}
\begin{equation}\label{exp +?}
\Bbb P(A,B;\mathcal W)\le \exp\left(-m\,c(\eps,\delta)\frac{\mu\nu}{\mu+\nu}\right) +\exp\bigl(-c(1-\eps)\ex[S]\bigr).
\end{equation}
\end{corollary}
\begin{proof} Immediate, by Lemma \ref{SapprE[S]} and \eqref{P(A,B:W)<E[exp(-)]}.
\end{proof}

\subsection{Bounding $\ex[S]$ from below.} To make the bound in Corollary \ref{break} usable, we need to find an explicit lower bound
for $\ex[S]$.  By \eqref{ESAW>}, we have
\begin{equation}\label{E[S]=prod}
\ex[S]=cm\Biggl(\sum_{a\in\mathcal A}a^{-1/2}\!\Biggr)\Biggl(\sum_{b\in\mathcal B}b^{-1/2}\!\Biggr).
\end{equation}
Using the bound
\[
j^{-1/2}\ge \psi(j)\boldsymbol{\boldsymbol{:}}=2\bigl((j+1)^{1/2}-j^{1/2}\bigr),
\]
we get a slightly cruder bound  
\[
\ex[S]\ge m f(\mathcal A,\mathcal B), \quad f(\mathcal A,\mathcal B):=\Biggl(\sum_{a\in\mathcal A}\psi(a)\!\Biggr)\Biggl(\sum_{b\in\mathcal B}\psi(b)\!\Biggr).
\]
Advantage of this replacement is the ease of summing $\psi(j)$ over uninterrupted intervals. Which pair $(\mathcal A^*,\mathcal B^*)$
minimizes $f(\mathcal A,\mathcal B)$? 
If we swap any two vertices $\a\in A^*$ and $\be\in B^*$,  then necessarily
\[
f\Bigl((\mathcal A^*\setminus\{\a\})\cup\{\be\}, (\mathcal B^*\setminus\{\be\})\cup\{\a\}\Bigr)\ge f(\mathcal A^*,\mathcal B^*),
\]
or equivalently
\begin{equation}\label{()x()<0}
\bigl(\psi(\a)-\psi(\be)\bigr)\left(\sum_{b\in \mathcal B^*\setminus\{\be\}}\psi(b)\,\,-\sum_{a\in \mathcal A^*\setminus\{\a\}}\psi(a)\!\right)\le 0.
\end{equation}
Suppose that
\begin{equation}\label{sum_b(ge)sum_a}
\sum_{b\in \mathcal B^*}\psi(b)\ge\sum_{a\in \mathcal A^*}\psi(a).
\end{equation}
If for some $\a\in \mathcal A^*$ and $\be\in \mathcal B^*$ we have $\psi(\a)>\psi(\be)$, then, by \eqref{()x()<0}, 
\[
\sum_{b\in B^*\setminus\{\be\}}\psi(b)\,\,-\sum_{a\in A^*\setminus\{\a\}}\psi(a)\le 0,
\]
which contradicts the combination of \eqref{sum_b(ge)sum_a} and  the assumption that $\psi(\a)>\psi(\be)$.
So the minimizer $(\mathcal A^*,\mathcal B^*)$ meets  the necessary condition: if \linebreak $\sum_{b\in \mathcal B^*}\psi(b)\ge (\le \text{resp.})\sum_{a\in \mathcal A^*}\psi(a)$, then $\mathcal A^*$ ($\mathcal B^*$ resp.) is the set of $\mu$ ($\nu$ resp.) largest elements in 
$\mathcal C=[n-\mu-\nu+1,n]$. 
So there are two possibilities for the pair $(\mathcal A^*,\mathcal B^*)$:
\begin{equation}\label{cases}
\begin{aligned}
&\text{\bf(1)\/}\,\, \mathcal A^*=\{n-\mu-\nu+1,\dots, n-\nu\},\quad \mathcal B=\{n-\nu+1,\dots, n\},\\
&\text{\bf (2)\/}\,\mathcal A^*=\{n-\mu+1,\dots, n\},\quad \mathcal B^*=\{n-\mu-\nu+1,\dots, n-\mu\}.
\end{aligned}
\end{equation}
In the first case, by telescoping the sums,
\begin{align}
&\qquad f(\mathcal A^*,\mathcal B^*)\boldsymbol{:}=\Biggl(\sum_{j=n-\mu-\nu+1}^{n-\nu}\psi(j)\!\Biggr)\cdot\Biggl(\sum_{j=n-\nu+1}^n\psi(j)\!\Biggr)\quad \notag\\
&=4\bigl(\sqrt{N-\nu}-\sqrt{N-\mu-\nu}\bigr)\bigl(\sqrt{N}-\sqrt{N-\nu}\bigr)=\boldsymbol{\boldsymbol{:}} h(\mu,\nu),\label{hnmunu}
\end{align}
$(N\boldsymbol{\boldsymbol{:}}=n+1)$, and  $f(A^*,\mathcal B^*)= h(\nu,\mu)$ in the second case.
After some algebra it follows that the first case holds for $\nu\le\mu$, and the second case for $\mu\le\nu$. So
\begin{equation}\label{g>}
f(\mathcal A^*,\mathcal B^*)=g(\mu,\nu)\boldsymbol{:}=\min\bigl(h(\mu,\nu),h(\nu,\mu)\bigr)=\left\{\begin{aligned}
&h(\mu,\nu),&\text{if }\mu\ge\nu,\\
&h(\nu,\mu),&\text{if }\mu\le\nu.\end{aligned}\right.
\end{equation}
Combining this formula with \eqref{P(Ssmall)<} and \eqref{E[S]=prod} we have proved
\begin{equation}\label{P(A,B)<(mu,nu)}
\Bbb P(A,B;\mathcal W)\le \exp\left(\!-m\,c(\eps,\delta)\frac{\mu\nu}{\mu+\nu}\!\right) +\exp\bigl(-m\,c(1-\eps)g(\mu,\nu)\bigr).
\end{equation}
By the union bound we arrive at 
\begin{theorem}\label{Emu,nu<} Let $\Bbb P(\mu,\nu)$ denote the probability that the event $\mathcal W$ holds, and that there exists a pair $(A,B)$ of vertex sets in $G_m^n$, 
($\mu+\nu= (1-\delta) n$), with no edge joining $A$ and $B$. Then
\begin{equation}\label{P(mu,nu)<}
\begin{aligned}
\Bbb P(\mu,\nu)&\le 2\binom{n}{\mu+\nu}
\binom{\mu+\nu}{\mu}\cdot \exp\bigl(-mH_{\eps,\delta}(\mu,\nu)\bigr),\\
H_{\eps,\delta}(\mu,\nu)&:=\min\left\{\frac{2\eps^2\sqrt{\delta}}{(1+\sqrt{\delta})^2}\frac{\mu\nu}{\mu+\nu}\,;\,c(1-\eps)g(\mu,\nu)\right\},\end{aligned}
\end{equation}
where $g(\mu,\nu)$ is given by \eqref{g>} and \eqref{hnmunu}, and $c=c(n)=0.5-o(1)$.
\end{theorem}
\begin{proof} Immediate, since the product of two binomals is the total number of ways to choose a pair of two subsets $A$ and $B$
of cardinality $\mu$ and $\nu$.
\end{proof} 
\subsection {\bf Example 1.\/} Let $\mu=\nu=\be n$, and $\be<1/2$. In this case it follows from Lemma \ref{Emu,nu<} that 
$P(\mu,\nu)\le \exp(-nJ_m(\be)+o(n))$ where $J_m(\be)=\min(J_{m,1}(\be), J_{m,2}(\be))$,
\begin{align*}
J_{m,1}(\be)&=I(\be) + m\cdot 4c(1-\eps)\bigl(\sqrt{1-\be}-\sqrt{1-2\be}\bigr)\bigl(1-\sqrt{1-\be}\bigr),\\
J_{m,2}(\be)&=I(\be)+m\cdot\frac{\eps^2\be\sqrt{1-2\be}}{(1+\sqrt{1-2\be})^2},\\
I(\be)&:=2\be\log\be+(1-2\be)\log(1-2\be).
\end{align*}
Maple shows that, for $\eps=6/7$, both $J_{16,1}(0.492)$ and $J_{16,2}(0.43)$ are positive. Therefore for all $m\ge 16$ there exist $\be_{m,1},\,\be_{m,2}\in (0,1/2)$ such that $J_{m,i}(\be)>0$ for $\be$ close enough to $\be_{m,i}$ from above. This means that,
for an arbitrarily small $\sigma>0$, with probability exponentially close to $1$, $G_m^n$ has no subsets $A$ and $B$, each of size above $(1+\sigma)\max(\be_{m,1},\be_{m,2})n$, and with no edge joining them.  
A closer look shows that 
\[
\be_{m,1}\sim \frac{4\log m}{m(1-\eps)},\quad \be_{m,2}=\exp\Bigl(-\eps^2m(1/8+o(1))),\quad m\to\infty.
\]
In particular, for every $\eps\in (0,1)$, $\be_{m,2}\ll \be_{m,1}$; so for $m$ large enough, whp there are no such pairs $(A,B)$ with each set of cardinality $\gtrsim n\frac{4\log m}{m}$.

To compare, it was proved in \cite{FGPR}  that such a threshold $\be^*(m)$ exists for $m\ge 24$,
and that $\be^*(m)\sim \frac{16 \log m}{m}$ for $m\to\infty$.  According to Lemma 9, part (i) in \cite{FGPR}, whenever such $\be$ exists, deterministically there is a path of length $(1-2\be)n$, at least. 
So our bound implies that whp $G_{16}^n$ already contains a path of length $\approx (1-2\be_{m,1})n \approx 0.016 n$. \\

{\bf Example 2.\/} Given $S\subset [n]$, let $N(S)$ be the set of outside neighbors of $S$. We say that $S$ vertex-expands at rate $\rho$
if $|N(S)|\ge\rho  |S|$. For a generic set $A$,
$|A|=\boldsymbol{:}\mu$, there is no edge between $A$ and $[n]\setminus (A\cup N(A))$. If $|N(A)|\le\rho |A|$, then $\boldsymbol|[n]\setminus (A\cup N(A))\boldsymbol|\ge n-(1+\rho)|A|$. Assuming that $n-(1+\rho)|A|>0$, there exists a set $B\subset [n]\setminus (A\cup N(A))$
with $|B|=\nu\boldsymbol{:}=n-(1+\rho)|A|$, which is not joined to $A$ even by a single edge. 

Therefore the probability that the event $\mathcal W$ holds and
there is a set $A$ with $|N(A)|\le \rho|A|$ is bounded above by $\Bbb P(\mu,\nu)$.
Denote $\mu/n=x$, then
$y\boldsymbol{\boldsymbol{:}}=\nu/n=1-(1+\rho)x$, so that $x<(1+\rho)^{-1}$.  If $\delta=x\rho$, then $\mu+\nu= (1-\delta)n$, so, assuming
that $x\rho$ is bounded away from $0$ and $1$, by Lemma \ref{Emu,nu<}, we have 
 \begin{align*}
 \Bbb P(\mu,\nu)&\le \exp(-nK_m(\rho,x)+o(n)),\\
 K_m(\rho,x)&=\min(K_{m,1}(\rho,x), K_{m,2}(\rho, x)), \quad x\in (0,(1+\rho)^{-1},
\end{align*}
 and explicitly
\begin{equation*}
\begin{aligned}
K_{m,1}(\rho,x)&=H(\rho,x)+m\cdot c(1-\eps)g(x,1-(1+\rho)x),\\
K_{m,2}(\rho,x)&=H(\rho,x)+m\cdot\frac{2\eps^2\sqrt{x\rho}}{(1+\sqrt{x\rho})^2}\frac{x(1-x(1+\rho))}{1-x\rho},\\
H(\rho,x)&=\rho x\log(\rho x)+x\log x+(1-(1+\rho)x)\log(1-(1+\rho)x),\\
g(x,y)&=\min(h(x,y),h(y,x))=\left\{\begin{aligned}
&h(x,y),&&\text{if }x\ge y,\\
&h(y,x),&&\text{if }x\le y,\end{aligned}\right.\\
h(x,y)&:=4\Bigl(\sqrt{1-x}-\sqrt{1-x-y}\Bigr)\Bigl(1-\sqrt{1-y}\Bigr).
\end{aligned}
\end{equation*}
It follows that
\begin{align*}
K_{m,i}(\rho,x)&\sim (1+\rho) x\log(\rho x),\quad x\to 0,\\
K_{m,i}(\rho,(1+\rho)^{-1})&=\frac{\rho}{1+\rho}\log\frac{\rho}{1+\rho}+\frac{1}{1+\rho}\log \frac{1}{1+\rho};
\end{align*}
so $K_{m,i}(\rho, x)<0$ for $x$ close to the extreme points $0$ and $(1+\rho)^{-1}$. Judging by Maple-aided computations, $K_{m,i}(\rho,x)$ either does not have positive zeros in $(0, (1+\rho)^{-1})$ or, {\it for $m$ large enough\/}, has two zeros, $x_i(m,\rho,\eps)<X_i(m,\rho,\eps)$, and $K_{m,i}(\rho,x)>0$
for $x\in I_i(m,\rho,\eps):= (x_i(m,\rho,\eps), X_i(m,\rho,\eps))$. It means that, for those $m$, $K_m(\rho,x)>0$ on $I(m,\rho,\eps):=I_1(m,\rho,\eps)\cap
I_2(m,\rho,\eps)$. So whp every set $A$ with $|A|/n\in I(m,\rho,\eps)$ vertex expands at rate $\rho$ at least. For instance, 
\begin{align*}
I(m=39, \rho=1,\eps=0.6)&\supset (0.288, 0.321),\\
I(m=500, \rho=1,\eps=0.6)&\supset (0.0155, 0.460),\\
I(m=65, \rho=2,\eps=0.6)&\supset (0.242, 0.252),\\
I(m=500, \rho=2,\eps=0.6)&\supset (0.0332, 0.305),
\end{align*}
In particular, for $m=500$, the right endpoint is close to $(1+\rho)^{-1}$, and the left endpoint is close to zero.
It is not difficult to show that, as $m\to\infty$,
\begin{align*}
X_r(m,\rho,\eps)&\sim \frac{\gamma_1(\rho)}{(1-\eps)^2}\cdot\left(\frac{\log m}{m}\right)^2,\quad \gamma_1(\rho)=4(1+\rho)^2(\sqrt{1+\rho}+\sqrt{\rho})^2,\\
x_2(m,\rho,\eps)&\sim \frac{\gamma_2(\rho)}{\eps^4}\cdot\left(\frac{\log m}{m}\right)^2,\qquad\,\gamma_2(\rho)=\frac{(1+\rho)^2}{4\rho},
\end{align*}
and
\begin{align*}
X_r(m,\rho,\eps)&=\frac{1}{1+\rho}-\frac{1+o(1)}{\sqrt{m(1-\eps)}}\,G(\rho),\quad G(\rho):=\frac{\rho^{1/4}}{(1+\rho)^{5/4}}H^{1/2}(\rho),\\
X_2(m,\rho,\eps)&=\frac{1}{1+\rho}-\frac{(1+o(1))\Bigl(1+\sqrt{\frac{\rho}{\rho+1}}\Bigr)^2}{2m\eps^2\sqrt{\rho(1+\rho)}}\,H(\rho),\\
H(\rho)&:=\frac{\rho}{1+\rho}\log\frac{\rho}{1+\rho}+\frac{1}{1+\rho}\log(1+\rho).
\end{align*}
Clearly $X_r(m,\rho,\eps)<X_2(m,\rho,\eps)$ for $m\to\infty$. We conclude that, for $m$ large,  whp every set $A$ of cardinality between 
\[
(1+\sigma)\max\bigl(X_r(m,\rho,\eps),x_2(m,\rho,\eps)\bigr)n\quad \text{and}\quad (1-\sigma) X_r(m,\rho,\eps)n
\]
vertex expands at rate $\rho$ at least. For large $m$ we get an asymptotically smallest $\max\bigl(X_r(m,\rho,\eps),x_2(m,\rho,\eps)\bigr)$
 by choosing $\eps$ equal to $\eps(\rho)$ the root of the equation
\[
\frac{\gamma_1(\rho)}{(1-\eps)^2}=\frac{\gamma_2(\rho)}{\eps^4}\Longrightarrow \eps(\rho)=2\left(1+\sqrt{1+16(\sqrt{\rho(\rho+1)}+\rho)}\right)^{-1}.
\]

{\bf Example 3.\/} It was proved in \cite{Hof} that a preferential attachment graph, more general than $G_m^n$, is whp connected for $m>1$,
which is the same as saying that every non-empty set $A\neq [n]$ is connected to its complement $A^c$ by at least one edge. We use 
Theorem \ref{Emu,nu<} to determine a range of $|A|$ for which the probability of no edge between $A$ and $A^c$ is exponentially
small.

Consider $A$ such that $x:=|A|/n < (2+\gamma)^{-1}$, for some $\gamma>0$. If no edge connects $A$ and $A^c$, then no
edge connects $A$ and any $B\subset A^c$ where $y:= |B|/n=1-\delta-x$, $\delta=\gamma x$, Clearly $x+y=1-\delta$, i.e. $\mu+\nu=(1-\delta)n$, where $\mu=|A|$, $\nu=|B|$, {\it and\/} $x<y$ since $x< (2+\gamma)^{-1}$.
By Lemma \ref{Emu,nu<}, we have 
\begin{align*}
\Bbb P(\mu,\nu)\le \exp\bigl(-n\mathcal K_m(x)+o(n)\bigr),\quad 
 \mathcal K_m(x)=\min\bigl(\mathcal K_{m,1}(x), \mathcal K_{m,2}(x)\bigr),
 \end{align*}
where
\begin{align*}
\mathcal K_{m,1}(x)&=\mathcal H(x)+m\cdot 4c(1-\eps)\bigl(\sqrt{1-x}-\sqrt{\gamma x}\bigr)\bigl(1-\sqrt{1-x}\bigr),\\
\mathcal K_{m,2}(x)&=\mathcal H(x)+m\cdot\frac{2\eps^2\sqrt{\gamma x}}{(1+\sqrt{\gamma x})^2}\frac{x(1-x(\gamma+1))}{1-\gamma x},\\
\mathcal H(x)&=\gamma x\log(\gamma x)+x\log x+(1-x(\gamma+1))\log(1-x(\gamma+1)),
\end{align*}
and $c=0.5+o(1)$. We want to find $x_m=x_m(\gamma)$ such that $\mathcal K_m(x)<0$ for $x\in (x_m(\gamma), (\gamma+2)^{-1}]$ if $m$ is sufficiently
large. 

{\bf(a)\/} Since $(1-z)\log(1-z)\ge -z$, $1-(1-x)^{1/2}\ge x/2$, and $x\le (\ga+2)^{-1}$, we have
\[
x^{-1}\mathcal K_{m,1}(x)\ge\gamma \log(\gamma x) +\log x -(1+\gamma)+2mc(1-\eps)
\frac{\sqrt{\ga+1}-\sqrt{\ga}}{\sqrt{\ga+2}}>0,
\]
provided that
\[
x>x_{m,1}(\ga):=\exp\Biggl(\!\frac{\ga+1-\ga\log\ga}{\ga+1}\!\Biggr)\cdot\exp\Biggl(\!-\frac{2mc(1-\eps)\Bigl(\sqrt{\ga+1}-\sqrt{\ga}\Bigr)}{(\ga+1)\sqrt{\ga+2}}\!\Biggr).
\]
and  $x_{m,1}(\ga)<(\ga+2)^{-1}$ for
\[
m>m_1(\ga):=\frac{\bigr[\ga+1-\ga\log\ga+(\ga+1)\log(\ga+2)\bigr]\sqrt{\ga+2}}{2c(1-\eps)\bigl(\sqrt{\ga+1}-\sqrt{\ga}\bigr)}.
\] 

 {\bf (b)\/} The function $(1-x(\ga+1))(1-\ga x)^{-1}(1+\sqrt{\ga x})^{-2}$ increases on $[0,(\ga+2)^{-1}]$, so that
\[
x^{-1}\mathcal K_{m,2}(x)> -(\ga+1)\log\frac{e}{x\ga}+m\eps^2 h(\ga) x^{1/2},\quad h(\ga):=\frac{(\ga+2)\sqrt{\ga}}{\bigl(\sqrt{\ga+2}+\sqrt{\ga}\bigr)^2}.
\]
Observe that $x^{-1/2}\log(e/x\ga)$ strictly decreases with $x$ increasing. So for
\[
m>m_2(\gamma):=\frac{\ga+1}{\eps^2h(\ga)}\cdot \left.x^{-1/2}h(x)\right|_{x=(\ga+2)^{-1}}
\]
there exists a unique root  $x_{m,2}(\ga)\in (0,(\ga+2)^{-1})$ of the equation 
\begin{equation}\label{xm2=}
x^{-1/2}\log\frac{e}{x\ga}=\frac{m\eps^2h(\ga)}{\ga+1},
\end{equation}
implying that $\mathcal K_{m,2}(x)>0$ for $x\in (x_{m,2}(\ga), (\ga+2)^{-1}]$. 

For $m>m(\ga):=\max(m_1(\ga),m_2(\ga))$, $x_m=\max(x_{m,1}(\ga),x_{m,2}(\ga))$ has the desired property: with probability
exponentially close to $1$ there are no isolated sets $A$ with $|A|/n\in(x_m(\ga), (\ga+2)^{-1})$. It follows from \eqref{xm2=} that for $m$ large
\[
x_{m,2}=\a(\gamma)\left(\frac{\log m}{\eps^2m}\right)^2 (1+O(\log\log m/\log m)),\quad \a(\ga):=\left(\frac{2(\ga+1)}{h(\ga)}\right)^2,
\]
meaning that $x_m= x_{m,2}$ for $m$ large enough.

\subsection{\label{revise}} The bound \eqref{P(A,B)<(mu,nu)} for $\Bbb P(A,B;\mathcal W)$ does not depend on choice of partition $C=A\cup B$, $C=[n-\mu-\nu+1,n]$. Is there a room for improvement?
 
Suppose $\mu\le\nu$, so that $A^*=[n-\mu+1,n]$, $B^*=[n-\mu-\nu+1, n-\mu]$. Generalizing the necessary condition \eqref{()x()<0},
let us swap $\mathcal A\subseteq A^*$ and $\mathcal B\subseteq B^*$, $|\mathcal A|=|\mathcal B|=r\le \min(\mu,\nu)$, so that
a new partition is $A=(A^*\setminus \mathcal A)\cup \mathcal B$, $B=(B^*\setminus \mathcal B)\cup \mathcal A$. Then, denoting
$x=\sum_{b\in\mathcal B} \psi(b)-\sum_{a\in \mathcal A}\psi(a)$,
\begin{align*}
f_C(A,B)&=f_C(A^*,B^*)+ x\left(\sum_{b\in B^*}\psi(b)-\sum_{a\in A^*}\psi(a)\right)-x^2\\
&=g_n(\mu,\nu) +x\left(\sum_{b\in B^*}\psi(b)-\sum_{a\in A^*}\psi(a)-x\right).
\end{align*}
Since $\psi(j)$ is decreasing, $\min\limits_{\mathcal A,\mathcal B}x$ ($\max\limits_{\mathcal A,\mathcal B}x$ resp.) is attained at $\mathcal A$ equal to the set of $r$ smallest (largest resp.) elements of $A^*$,
and $\mathcal B$ equal to the set of $r$ largest (smallest resp.) elements of $B^*$. So, telescoping the resulting sums and denoting $n+1=N$,
we have
\begin{equation}\label{xi=}
\begin{aligned}
X_r&=\min_{\mathcal A,\,\mathcal B} x=\sum_{j=n-\mu-r+1}^{n-\mu} \psi(j)-\sum_{j=n-\mu+1}^{n-\mu+r}\psi(j)\\
&=2\bigl(\sqrt{N-\mu}-\sqrt{N-\mu-r}-\sqrt{N-\mu+r}+\sqrt{N-\mu}\bigr);\\
x_2&=\max_{\mathcal A,\,\mathcal B}x=\sum_{j=n-\mu-\nu+1}^{n-\mu-\nu+r}\psi(j)-\sum_{j=n-r+1}^n\psi(j)\\
&=2\bigl(\sqrt{N-\mu-\nu+r}-\sqrt{N-\mu-\nu}-\sqrt{N}+\sqrt{N-r}\bigr).
\end{aligned}
\end{equation}
Clearly $x_i$ depend on $r$, the number of elements from $A$ swapped for the same number of elements from $B$. $x_i(0)=0$, and as functions of a continuously varying $r\in [0,\min(\mu,\nu)]=[0,\mu]$, they satisfy $(x_2-X_r)'_r>0$. Therefore $x_2(r)>X_r(r)$ for $r>0$. Further, $f_C(A,B)$ is a concave function of $x$, so for each $r$ $f_C(A,B)$ attains its minimum value at either $X_r$ or $x_2$. Now
\begin{align*}
&x_2\Biggl(\sum_{b\in B^*}\psi(b)-\sum_{a\in A^*}\psi(a)-x_2\!\!\Biggr)-X_r\Biggl(\sum_{b\in B^*}\psi(b)-\sum_{a\in A^*}\psi(a)-X_r\!\!\Biggr)\notag\\
&\qquad\qquad=(x_2-X_r)\left(\sum_{b\in B^*}\psi(b)-\sum_{a\in A^*}\psi(a) -(x_2+X_r)\right)\notag\\
=&2(x_2-X_r)\Bigl(\sqrt{N-\mu}-\sqrt{N-\mu-\nu}-\sqrt{N}+\sqrt{N-r}-(x_2+X_r)\Bigr)\notag\\
&\qquad\qquad\qquad\quad (\text{plugging in }X_r,\,x_2\text{ from }\eqref{xi=})\notag\\
&\qquad\qquad\qquad\qquad=2(x_2-X_r)D(\mu,\nu,r);\notag
\end{align*}
here
\begin{align}
D(\mu,\nu,r)&\boldsymbol{\boldsymbol{:}}=\sqrt{N-\mu-r}+\sqrt{N-\mu+r}\label{Dn(mu,nu,r)=}\\
&\quad-\sqrt{N-\mu-\nu+r}-\sqrt{N-r}.\notag
\end{align}
Since $x_2-X_r>0$, we see that $x_2$ ($X_r$ resp.) is the minimum point if $D(\mu,\nu,r)<0$ (if $D(\mu,\nu,r)>0$ resp.). Now
\begin{align}
&d_{i}(\mu,\nu,r)\boldsymbol{\boldsymbol{:}}= x_i\!\!\left(\sum_{b\in B^*}\psi(b)-\sum_{a\in A^*}\psi(a)-x_i\!\!\right)\!=x_i\!\!\left(\sum_{B^*\setminus\mathcal B}\psi(b)-\!\!\!\sum_{a\in A^*\setminus\mathcal A}\psi(a)\!\!\right)\notag\\
&\qquad\qquad\qquad\quad (\text{plugging in }x_i\text{ and telescoping the sums})\notag\\
&\label{dni=expl}\qquad\,\,\,=\!\left\{\begin{aligned}
&\!4\Bigl(2\sqrt{N-\mu}-\!\sqrt{N-\mu-r}-\!\sqrt{N-\mu+r}\Bigr)&\\
&\times\Bigl(\sqrt{N-\mu-r}-\sqrt{N-\mu-\nu}-\sqrt{N}+\sqrt{N-\mu+r}\Bigr);\!&i=1;\\
&\\
&4\bigl(\sqrt{N-\mu-\nu+r}-\sqrt{N-\mu-\nu}-\sqrt{N}+\sqrt{N-r}\bigr)&\\
&\!\times\Bigl(2\sqrt{N-\mu}-\!\sqrt{N-\mu-\nu+r}-\!\sqrt{N-r}\Bigr),\!&i=2.
\end{aligned}\right.\notag\\
\end{align}
Introduce $g(\mu,\nu,r)=\min\bigl\{f_C(A,B)\bigr\}$ over all partitions $C=A\cup B$ where $A$ and $B$ are obtained from $A^*$ and $B^*$ by replacing some $r$ elements in $A^*$ with $r$ elements from $B^*$. We conclude that, for $\mu\le \nu$,
\begin{equation}\label{gn(mu,nu,r)=}
\begin{aligned}
&\qquad\quad g(\mu,\nu,r)= g(\mu,\nu)+d(\mu,\nu,r),\\
&d(\mu,\nu,r)\boldsymbol{:}=\left\{\begin{aligned}
&d_{1}(\mu,\nu,r),&\text{ if }D(\mu,\nu,r)\ge 0,\\
&d_{2}(\mu,\nu,r),&\text{ if }D(\mu,\nu,r)\le 0,\end{aligned}\right.
\end{aligned}
\end{equation}
see \eqref{Dn(mu,nu,r)=} for $D(\mu,\nu,r)$. So the counterpart of the bound \eqref{P(A,B)<(mu,nu)} is 
\begin{equation}\label{P(A,B)<(mu,nu)better}
\Bbb P(A,B;\mathcal W)\le \exp\Biggl(-m\,c(\eps,\delta)\frac{\mu\nu}{\mu+\nu}\Biggr) +\exp\bigl(-m\,c(1-\eps)g(\mu,\nu,r)\bigr).
\end{equation}
We have proved an extension of Theorem \ref{Emu,nu<}:

\begin{theorem}\label{ExtEmu,nu<} Let $\Bbb P(\mu,\nu,r)$, ($r\le \mu\le\nu$) denote the probability that the event $\mathcal W$ holds, and that there exists a pair $(A,B)$ of vertex sets in $G_m^n$, 
($|A|=\mu$, $|B|=\nu$,  $\mu+\nu=(1-\delta) n$), with exactly $\mu-r$ elements of $A$ being among the $\mu$ largest elements of $A\cup B$, such that  no pair $(a,b)$ \, ($a\in A, b\in B$), is an edge in $G_m^n$. Then
\begin{align*}
 &\Bbb P(\mu,\nu,r)\le \binom{n}{\mu+\nu}
\binom{\mu}{r}\binom{\nu}{r}\\
&\times \left[\exp\Bigl(-m\,c(\eps,\delta)\frac{\mu\nu}{\mu+\nu}\Bigr) +\exp\bigl(-m\,c(1-\eps)g(\mu,\nu,r)\bigr)\right],
\end{align*}
where $g(\mu,\nu,r)$ is given by \eqref{gn(mu,nu,r)=}, and $c=c(n)=0.5-o(1)$.
\end{theorem}

\section{Maximal recursive trees} 
Each vertex $v\in [n]$ is a root of a tree $T$ such that on every path going away from $v$ the vertices increase. In other words, $T$ is
a recursive tree on $V(T)$. $T$ is maximal if no outside vertex selects a vertex in $V(T)$. The size of a maximal tree rooted at $v$ can be viewed as an influence measure of $v$. 
\begin{theorem}\label{maxtree} {\bf(1)\/} Whp there are no maximal recursive trees with vertex sets from the first $\mu=n^{\frac{m}{m+4}-\eps}$ vertices.
{\bf (2)\/} For $\mu\to\infty$ and $\mu=o(n)$, whp no subset of $[\mu]$ of cardinality comparable to $\mu$ is a vertex set of a maximal recursive tree.
\end{theorem}

\begin{proof} 
Given $A=\{a_1<a_2<\cdots<a_{\nu}\}$, let $T$ be a generic recursive tree on $A$. Introduce $\Bbb P(T\boldsymbol|\,\bold W)=
\Bbb P(T\text{ is a maximal recursive tree rooted at } a_1\boldsymbol|\,\bold W)$. Conditioned on $\bold W$, the events $a\leftarrow b$, ($a<b, (a,b)\in E(T)$), are independent among themselves and from the events $a\not\leftarrow b$, ($a\in A$, $b\notin A$), and the latter events are negatively associated among themselves. Further,
recall that for $a<b$, the event $a\leftarrow b$ means that at least one
of the right endpoints $R_{m(b-1)+1},\dots, R_{mb}$ has its left endpoint between $R_{m(a-1)}$ and $R_{ma}$. Therefore
\begin{align}
&\qquad\quad\Bbb P(a\leftarrow b\boldsymbol|\,\bold W)\le \sum_{i=m(b-1)+1}^{mb}\int_{R_{m(a-1)}}^{R_{ma}}\frac{dx}{2(R_i x)^{1/2}}\notag\\
&\quad =\sum_{i=m(b-1)+1}^{mb}\frac{R^{1/2}_{ma}-R^{1/2}_{m(a-1)}}{R_i^{1/2}}
\le m\,\frac{R^{1/2}_{ma}-R^{1/2}_{m(a-1)}}{R^{1/2}_{m(b-1)}}\notag\\
&\quad\qquad\qquad\qquad=\frac{m(W_{ma}^{1/2}-W_{m(a-1)}^{1/2})}{W^{1/2}_{m(b-1)}}.\label{basic2}
\end{align}
Using \eqref{basic} and \eqref{basic2} we obtain
\begin{align}
&\qquad\qquad\Bbb P(T\boldsymbol|\,\bold W)\le \prod_{(a, b)\in E(T)\atop a<b}\!\!\!\Bbb P(a\leftarrow b\boldsymbol|\,\bold W)\cdot\prod_{a\in A, c\notin A\atop
a<c}\!\!\Bbb P(a\not\leftarrow c\boldsymbol|\,\bold W)\notag\\
&\qquad\le \prod_{(a, b)\!\!\in E(T)\atop a<b}\!\!\frac{m(W_{ma}^{1/2}-W_{m(a-1)}^{1/2})}{W^{1/2}_{m(b-1)}} \cdot\prod_{a\in A,\, c\notin A\atop
a<c}\!\! \exp\Biggl(\!-m\frac{W_{ma}^{1/2}-W_{m(a-1)}^{1/2}}{W_{mc}^{1/2}}\!\Biggr)\notag\\
&\qquad\qquad=\prod_{j=2}^{\nu} W^{-1/2}_{m(a_j-1)}\cdot \prod_{j=1}^{\nu-1}\Bigl(m(W_{ma_j}^{1/2}-W_{m(a_j-1)}^{1/2})\Bigr)^{d_{in}(a_j)}\notag\\
&\qquad\times\exp\Biggl[-m\sum_{j=1}^{\nu-1}\Bigl( W_{ma_j}^{1/2}-W_{m(a_j-1)}^{1/2} \Bigr)\,\cdot\Biggl(\sum_{c\notin A\boldsymbol{:}\, c>a_j}W_{mc}^{-1/2}\Biggr)\Biggr];\label{P(T|W)<}
\end{align}
here $d_{in}(a_j)$ is the in-degree of $a_j$ in $T$, i.e. the number of $a_j$'s neighbors in $T$, which are one edge further from the root $a_1$. Notice that the first product and the exponential factor do not depend on choice of the recursive tree $T$. The next Lemma allows
us to sum the second product over all $(\nu-1)!$ choices of the tree $T$. 
\begin{lemma}\label{rectree} Given $\bold d(\nu)=\{d_1\dots, d_{\nu}\}$, let $N(\bold d(\nu))$ be the total number of recursive trees on $[\nu]$
with the in-degree sequence $\bold d(\nu)$. If $N(\bold d(\nu))\neq 0$ then $d_{\nu}=0$ and $d_{\nu-1}\in \{0,1\}$. Let $\bold z(\nu-1)=\{z_1,\dots,z_{\nu-1}\}$ be a $(\nu-1)$-tuple of indeterminants, and 
\[
F(\bold z(\nu-1)):=\sum_{\bold d(\nu)} N(\bold d(\nu)) \prod_{i\in [\nu-1]}\!\! z_i^{d_i}.
\]
Then
\[
F(\bold z(\nu-1))=\prod_{j\in[\nu-1]}\Biggl(\sum_{i\in [j]}z_i\Biggr).
\]
\end{lemma}
\begin{proof} Deleting vertex $\nu$ we obtain a recursive tree on $[\nu-1]$. Therefore, for $d_{\nu-1}=0$,  we have a recursion
\begin{align*}
&\quad\, N(\bold d(\nu))=\sum_{j\in [\nu-2]} \Bbb I(d_j\ge 1)\cdot N(\bold d^{(j)}(\nu-1)),\\
&\bold d^{(j)}(\nu-1):=\{d_1,\dots,d_{j-1}, d_j-1, d_{j+1},\dots,d_{\nu-2}, 0\},
\end{align*}
and if $d_{\nu-1}=1$, then
\[
N(\bold d(\nu))=N(\bold d^{(\nu-1)}(\nu-1)),\quad \bold d^{(\nu-1)}(\nu-1):=\{d_1,\dots,d_{\nu-2}, 0\}.
\]
So
\begin{multline*}
F(\bold z(\nu-1))=\sum_{\bold d(\nu)\boldsymbol{:}\,d_{\nu-1}=0}\,\prod_{i\in[\nu-2]}z_i^{d_i}\sum_{j\in [\nu-2]}\Bbb I(d_j\ge 1)N(\bold d^{(j)}(\nu-1))\\
+\sum_{\bold d(\nu)\boldsymbol{:}\,d_{\nu-1}=1}\, \prod_{i\in [\nu-1]}z_i^{d_i} N(\bold d^{(\nu-1)}(\nu-1))\\
=\Biggl(\sum_{j\in [\nu-2]}z_j\Biggr) F(\bold z(\nu-2))+ z_{\nu-1}F(\bold z(\nu-2))
=\Biggl(\sum_{j\in [\nu-1]}z_j\Biggr)F(\bold z(\nu-2)).
\end{multline*}
\end{proof}
Let $\Bbb P(A\boldsymbol|\,\bold W)$ denote the conditional probability that $G_m^n$ contains a recursive tree spanning the set $A$,
and this tree is maximal in $G_m^n$.
Then $\Bbb P(A\boldsymbol|\,\bold W)\le \sum_T\Bbb P(T\boldsymbol|\,\bold W)$, and applying Lemma \ref{rectree} we obtain
\begin{equation}\label{Pi(A)<}
\begin{aligned}
&\Bbb P(A\boldsymbol|\,\bold W)\le \prod_{j=2}^{\nu} W^{-1/2}_{m(a_j-1)}\cdot \prod_{j\in [\nu-1]}\Biggl(m\sum_{i\in [j]}
\Bigl(W_{ma_i}^{1/2}-W_{m(a_i-1)}^{1/2}\Bigr)\Biggr)\\
&\times\exp\Biggl[-m\sum_{j\in[\nu-1]}\Bigl( W_{ma_j}^{1/2}-W_{m(a_j-1)}^{1/2} \Bigr)\,\cdot\Biggl(\sum_{c\notin A\boldsymbol{:}\, c>a_j}W_{mc}^{-1/2}\Biggr)\Biggr].
\end{aligned}
\end{equation}
Here, since $a_{k-1}\le a_k-1$, we have
\[
\sum_{i\in [j]}\Bigl(W_{ma_i}^{1/2}-W_{m(a_i-1)}^{1/2}\Bigr)\le \sum_{i\in [j]}\Bigl(W_{ma_i}^{1/2}-W_{ma_{i-1}}^{1/2}\Bigr)=W_{ma_j}^{1/2},
\]
so that 
\begin{align*}
B_1(n:A)&:=\prod_{j=2}^{\nu} W^{-1/2}_{m(a_j-1)}\cdot \prod_{j\in [\nu-1]}\Biggl(m\sum_{i\in [j]}
\Bigl(W_{ma_i}^{1/2}-W_{ma_{i-1})}^{1/2}\Bigr)\Biggr)\\
&\le m^{\nu-1}\prod_{j=2}^{\nu} W^{-1/2}_{ma_{j-1}}\cdot\prod_{j\in [\nu-1]}W_{ma_j}^{1/2}=m^{\nu-1}.
\end{align*}

Turn to $B_2(n;A)$, which is the exponential factor in \eqref{Pi(A)<}. Notice that
\[
\sum_{c\notin A\boldsymbol{:}\, c>a_j}W_{mc}^{-1/2}= \sum_{c>a_j}W_{mc}^{-1/2}-\sum_{i>j}W_{ma_i}^{-1/2},
\]
and, summing by parts,
\begin{align*}
&\sum_{j\in [\nu-1]}\Bigl( W_{ma_j}^{1/2}-W_{m(a_j-1)}^{1/2} \Bigr)\cdot\Biggl(\sum_{i>j}W_{ma_i}^{-1/2}\Biggr)
&=\sum_{j\in [\nu-1]}\!\!W_{ma_j}^{1/2} \cdot W_{ma_{j+1}}^{-1/2}\le \nu-1.
\end{align*}
Consequently 
\begin{align*}
&\qquad\quad\qquad B_2(n;A)\le e^{m\nu} e^{-m S(n;A)},\\
&S(n;A):=\sum_{j\in[\nu-1]}\Bigl( W_{ma_j}^{1/2}-W_{m(a_j-1)}^{1/2} \Bigr)\cdot\Biggl(\sum_{c>a_j}W_{mc}^{-1/2}\Biggr).
\end{align*}
Recall now that $W_{ma_j}^{1/2}-W_{m(a_j-1)}^{1/2}\ge \Omega_{ma_j} W^{-1/2}_{ma_j}$. Further, if 
$\la=\la(n)\to\infty$ however slowly, then by \eqref{C4}
\[
\Bbb P(W_i\le \la i,\, \forall\,i\in [mn])\ge 1-e^{-\phi(\la)}\to 1.
\]
Let us call $\Lambda$ this likely event. We see that on $\Lambda$
\begin{align*}
m S(n;A)&\ge \la^{-1} \sum_{j\in [\nu-1]} \Omega_{ma_j} a_j^{-1/2}\sum_{c>a_j} c^{-1/2}\\
&\ge \la^{-1} \sum_{j\in [\nu-1]} \Omega_{ma_j} a_j^{-1/2}2\bigl(\sqrt{n+1}-\sqrt{a_j+1}\bigr)\\
&\ge \la^{-1}\sum_{j\in [\nu-1]}\Omega_{ma_j}\left((n/a_j)^{1/2}-1\right).
\end{align*}
In summary, on the event $\Lambda$, the bound \eqref{Pi(A)<} has become
\[
\Bbb P(A\boldsymbol|\,\bold W)\le m^{\nu-1} e^{m\nu}
\exp\Biggl(\!-\la^{-1}\!\!\sum_{j\in [\nu-1]}\!\!\Omega_{ma_j}\left((n/a_j)^{1/2}-1\right)\!\!\Biggr).
\]
Denote $A\cap \Lambda$ the event that $A$ is spanned by a maximal recursive tree and $\Lambda$ holds. Since the $\nu$ variables $\Omega_{mj}$ are independent, each distributed as $\sum_{t\in [m]} w_t$,  the last bound yields 
\begin{align*}
\Bbb P(A\cap \Lambda)&\le  m^{\nu-1} e^{m\nu} \prod_{j\in [\nu-1]}\ex\Bigl[-\la^{-1}\Bigl((n/a_j)^{1/2}-1\Bigr) \sum_{t\in [m]}w_t\Bigr]\\
&= m^{\nu-1} e^{m\nu}\prod_{a\in A, a\neq \max(A)}\Bigl(1+\la^{-1}\bigl((n/a)^{1/2}-1\bigr)\Bigr)^{-m}.
\end{align*}
For $\mu\in [\nu,n]$, let $\Bbb P(\nu,\mu)$ be the probability that there exists a maximal recursive tree of size $\nu$ formed by
vertices from $\mu$. 
Summing the last bound over all  $A\subset [\mu]$ of cardinality $\nu$ we obtain
\begin{equation}\label{P(nu,mu)<}
\Bbb P(\nu,\mu)\le \Bbb P(\Lambda^c) +\mu\frac{m^{\nu-1} e^{m\nu}}{(\nu-1)!}\Biggl(\sum_{a\in [\mu]}\Bigl(1+\la^{-1}\Bigl((n/a)^{1/2}-1\Bigr)\Bigr)^{-m}\Biggr)^{\nu-1}.
\end{equation}
Here $\Bbb P(\Lambda^c)\to 0$ for $\la=\la(n)\to\infty$ however slowly. 

Consider the rest of the RHS expression, call it $Q(\nu,\mu)$. Suppose that $\mu=o(n^{\frac{m}{m+4}})$. By Theorem \ref{degs,m}, it suffices to consider $\nu>1$ only.  Let us choose $\la=\la(n)\to\infty$ such that $\la^2\mu^{\frac{m+2}{m}}=o(n)$. Then the sum in \eqref{P(nu,mu)<} is asymptotic to
\begin{align*}
\sum_{a\in\mu}\Bigl(\frac{\la^2a}{n}\Bigr)^{m/2}\sim \frac{2}{m+2}\left(\frac{\la^2\mu^{\frac{m+2}{m}}}{n}\right)^{m/2}\!\!\!\!\!\to 0,\quad n\to\infty.
\end{align*}
As $(\nu-1)!\ge 0.5 (\nu/e)^{\nu-1}$, {\it and\/} $\mu=o(n^{\frac{m}{m+4}})$, it follows easily that $\sum_{\nu=2}^{\mu}$\linebreak $Q(\nu,\mu)
\to 0$, so that
whp the set $[\mu]$ does not contain any maximal recursive tree. 

If $\mu=o(n)$ and $\mu\to\infty$, then, with $C=\frac{3me^{m+1}}{m+2}$,  we have
\begin{align*}
\mu\frac{m^{\nu-1} e^{m\nu}}{(\nu-1)!}\Biggl(\sum_{a\in [\mu]}\Bigl(\!1+\la^{-1}\Bigl(\!(n/a)^{1/2}-1\!\Bigr)\Bigr)^{-m}\!\Biggr)^{\nu-1}\!\!\!\!
\le\! \mu\nu\left(\frac{C\mu}{\nu}\Bigl(\frac{\la^2\mu}{n}\Bigr)^{m/2}\right)^{\nu-1}\!\!,
\end{align*}
if $\nu>(C+\eps)\mu\bigl(\la^2\mu/n\bigr)^{m/2}$. Choosing $\la\to\infty$ such that $\la^2\mu=o(n)$, we obtain that whp the set $[\mu]$ does not contain maximal recursive trees of sizes $\nu=\Theta(\mu)$.
\end{proof}
How large is the largest recursive subtree? To have a recursive subtree of size $n$ it is necessary and sufficient that the event
$\bigcap_{j>1} \mathcal L_{n,j}$ holds, where $\mathcal L_{n,j}$ is the event ``there are at most
$m-1$ loops at vertex $j$''.  Since the total vertex degree of each $G_m^t$ is $2mt$, the events $\mathcal L_{n,j}$ are independent, and
\begin{equation}\label{Hof}
\Bbb P(\mathcal L^c_{n,j})=\prod_{k=0}^{m-1}\frac{2k+1}{2(j-1)m +2k+1}=O(j^{-m}),
\end{equation}
implying that
\[
\Bbb P(\mathcal L_n)=\prod_{j=2}^n\Biggl(1-\prod_{k=0}^{m-1}\frac{2k+1}{2(j-1)m +2k+1}\Biggr)=\prod_{j=2}^n \Bigl(1-O(j^{-m})\Bigr).
\]
So, for $m>1$, $\liminf_{n\to\infty}\Bbb P(\mathcal L_n)>0$, whence with probability bounded away from $0$ the largest recursive subtree contains all $n$ vertices.  
 
Let an integer $\omega=\omega(n)\to\infty$ however slowly. Let $\mathcal L_{n,>\omega}$ be the event ``every vertex in $[n]\setminus \omega$
is on an increasing path emanating from a vertex in $[\omega]$''. On $\mathcal L_{n,>\omega}$ the set $[n]\setminus [\omega]$ is a union
of (not necessarily disjoint) sets $A_1,\dots, A_{\omega}$: vertices from $A_j$ are on increasing paths from $j$, and $\max |A_j|\ge \frac{n-\omega}{\omega}$. Furthermore $\mathcal L_{n,>\omega}=\bigcap_{j>\omega}\mathcal L_{n,j}$,
and therefore
\begin{align}
\Bbb P(\mathcal L_{n,>\omega})\!&=\!\!\prod_{j=\omega+1}^n\Biggl(\!1-\prod_{k=0}^{m-1}\frac{2k+1}{2(j-1)m +2k+1}\!\Biggr)\notag\\
&=\!\prod_{j=\omega+1}^n \!\Bigl(1-O(j^{-m})\!\Bigr)=1-O(\omega^{-m+1})\to 1,\label{omega->}
\end{align}
as $n\to\infty$. Therefore whp there exists a recursive tree of size $(1-o(1))\frac{n-\omega}{\omega}\sim \frac{n}{\omega}$ at least, i.e. whp
the size of the largest recursive tree is of order $n$.\\

\subsection{Connectedness of $G_m^n$ for $m>1$.}  
The equation \eqref{Hof} is a particular case of the formula for a more general preferential attachment graph $P\!A_n(m,\delta)$; it was used in \cite{Hof} to show that for $m>1$ whp $P\!A_n(m,\delta)$ is connected.
Of course, whp connectedness of $G_m^n$ followed directly from the likely upper bound for the diameter of $G_m^n$ established in \cite{BolRio}. A rate with which the high probability converges to $1$ remained undetermined. 

\noindent Let $\log^{(t)}$ denote the $t$-fold composition of $\log$ with itself; so $\log^{(1)}=\log$.

\begin{theorem}\label{connect} For $m>1$,  and arbitrarily small $\eps>0$,
\[
\Bbb P(G_m^n \text{ is connected})= 1-O\Biggl(\!\Biggl(\frac{(\log^{(2)}n)^{1+\eps}}{\log n}\Biggr)^{\frac{m-1}{3}}\Biggr).
\]
\end{theorem}
\noindent {\bf Note.\/} In contrast, $\Bbb P(G_1^n \text{ is connected})\sim 0.5(\pi/n)^{1/2}$.
\begin{proof} Pick $\omega=\lfloor \log^{\be} n\rfloor$, $\be>0$.  {\bf (1)\/} By \eqref{omega->}, with probability 
$1-O(\omega^{-(m-1)})$, every vertex in $[n]\setminus [\omega]$ is on an increasing path starting from a vertex in $[\omega]$. For a properly chosen $\be$, we will show {\bf (2)\/} that, if $\bold W$ is likely, then for every  two vertices $a_1\neq a_2\in [\omega]$,
with (conditional) probability $1-o(1)$ there exists a vertex $b\in [n]\setminus [\omega]$ (termed ``connector'' in \cite{Hof}) such that $a_1\leftarrow b$ and $a_2\leftarrow b$, and {\bf (3)\/} that
the $o(1)$ deficits for all  $\binom{\omega}{2}$ pairs $(a_1,a_2)$ add up to $o(1)$ still. The three items put together will imply the claim. 

Conditioned on $\bold W$, the $n-\omega$ events $\{a_1\leftarrow b,\,a_2\leftarrow b\}_{b\in [n]\setminus [\omega]}$ are independent. Let $X(\bold a)$ be the sum of these events indicators, so that $\ex[X(\bold a)\boldsymbol|\,\bold W]=\sum_{\bold b}\Bbb P(a_1\leftarrow b,\,a_2\leftarrow b\boldsymbol|\,\bold W)$. The vertex $b$ contains $m$ right endpoints $R_{m(b-1)+1},\dots, R_{mb}$, and
an event $a\leftarrow b$ holds if at least one of the left endpoints $\ell_i$, $(i\in [m(b-1)+1, mb])$ is located in the interval $(m(a-1), ma]$. Therefore
\[
\Bbb P(a_1\leftarrow b,\,a_2\leftarrow b\boldsymbol|\,\bold W)\ge\frac{W^{1/2}_{ma_1}-W^{1/2}_{m(a_1-1)}}{W^{1/2}_{m(b-1)+1}}
\cdot\frac{W^{1/2}_{ma_2}-W^{1/2}_{m(a_2-1)}}{W^{1/2}_{mb}},
\]
whence
\[
\ex[X(\bold a)\boldsymbol|\,\bold W]\ge \prod_{i=1}^2 \bigl(W^{1/2}_{ma_i}-W^{1/2}_{m(a_i-1)}\bigr)\cdot\sum_b\,\bigl(W_{m(b-1)+1}W_{mb}\bigr)^{-1/2}.
\]
By \eqref{C4''}, for $\sigma\in (0,1/2)$, we have
\[
\Bbb P\Bigl(|W_{\nu}-\nu|\le \nu^{1-\sigma},\,\forall \nu \ge \omega\Bigr)\ge 1-\exp\bigl(-\Theta(\omega^{1-2\sigma})\bigr).
\]
Therefore, with probability $1-\exp\bigl(-\Theta(\omega^{1-2\sigma})\bigr)$,
\begin{align*}
\sum_b\bigl(W_{m(b-1)+1}W_{mb}\bigr)^{-1/2}&\ge \!\bigl(1+O(\omega^{-\sigma})\bigr)\!\!\sum_{b\in [n]\setminus [\omega]}\!\!b^{-1}
= \bigl(1+O(\omega^{-\sigma})\bigr) \log n.
\end{align*}
Next
\[
W^{1/2}_{ma}-W^{1/2}_{m(a-1)}\ge \frac{\Omega_{ma}}{2W^{1/2}(ma)}\ge \frac{\Omega_{ma}}{2W^{1/2}(m\omega)}.
\]
Here $\Bbb P(W_{m\omega}\le 2m\omega)=1-e^{-\Theta(\omega)}$, and using independence of $\Omega_{ma}$, we have
\begin{align*}
\Bbb P\Bigl(\min_{a\le \omega}\Omega_{ma}\ge x\Bigr)&=\Biggl(1-\int_{z\le x}\frac{z^{m-1}}{\Gamma(m)}\,e^{-z}\,dz\Biggr)^{\omega}
=\exp\bigl(-\Theta(\omega^{-mD})\bigr)\to 1,
\end{align*}
if $x= \omega^{-1/m-D}$, ($D>0$). Therefore
\begin{equation}\label{use?}
\min_{\bold a}\ex\bigl[X(\bold a)\boldsymbol|\,\bold W\bigr]\ge (8.1m)^{-1}\frac{\log n}{\omega^{\frac{m+2}{m}+2D}},
\end{equation}
 with probability at least
\[
\min\bigl(1-e^{-\Theta(\omega^{1-2\sigma})}\bigr);\, 1-e^{-\Theta(\omega)} ; \,1-O(\omega^{-mD})\bigr)=1-O(\omega^{-mD}).
\]
($\omega$ growing faster than a power of logarithm would have rendered  the lower bound \eqref{use?} useless.)
Select
\[
D=\frac{m-1}{m},\quad \beta=\frac{1}{3}\left(1-\frac{(1+\eps)\log^{(3)} n}{\log^{(2)}n}\right).
\]
Then $\omega^{-mD}=\omega^{-(m-1)}$ and
\begin{align*}
\frac{\log n}{\omega^{\frac{m+2}{m}+2D}}&= (\log n)^{1-3\be}=(\log n)^{\frac{(1+\eps)\log^{(3)}n}{\log^{(2)} n}}\\
&=\exp\bigl((1+\eps)\log^{(3)}n\bigr)=(\log^{(2)}n)^{1+\eps},
\end{align*}
so that $\min_{\bold a}\ex\bigl[X(\bold a)\boldsymbol|\,\bold W\bigr]\ge (8.1 m)^{-1}(\log^{(2)}n)^{1+\eps}$.
Applying \eqref{C3} to $X(\bold a)$, we have: with probability $1-O(\omega^{-mD})=1-O(\omega^{-(m-1)})$,
\[
\Bbb P\bigl(X(\bold a)\le (9m)^{-1}(\log^{(2)} n)^{1+\eps}\boldsymbol|\,\bold W\bigr) \le \exp\bigl(-\Theta((\log^{(2)} n)^{1+\eps})\bigr).
\]
Using the union bound, we obtain: with probability $1-O(\omega^{-(m-1)})$,
\begin{equation*}
\begin{aligned}
&\Bbb P\bigl(\exists\bold a:X(\bold a)\le (9m)^{-1}(\log^{(2)} n)^{1+\eps}\boldsymbol|\,\bold W\bigr)\!\!
\le \binom{\omega}{2}\!\exp\bigl(-\Theta((\log^{(2)} n)^{1+\eps})\bigr)\\
&\qquad\qquad\quad\le\exp\bigl(2\be\log^{(2)} n-\Theta((\log^{(2)} n)^{1+\eps})\bigr)\ll \omega^{-(m-1)}.
\end{aligned}.
\end{equation*}
 ``Unconditioning'' we get
\[
\Bbb P\bigl(\exists\,\bold a:\,X(\bold a)\le (9m)^{-1}(\log^{(2)} n)^{1+\eps} \bigr) =O\bigl(\omega^{-(m-1)}\bigr).
\]
Thus with probability $1-O(\omega^{-(m-1)})$ for every two vertices $a_1,\,a_2\in [\omega]$ there exists a connector $b\in [n]\setminus [\omega]$, {\it and\/}, as we mentioned at the outset, each of the vertices $b\in  [n]\setminus [\omega]$ is on an increasing path going out of a vertex $a\in [\omega]$. Therefore whp $G_m^n$ is connected. It remains to notice that
\begin{align*}
\omega^{-(m-1)}&=\exp\bigl(-(m-1)\be\log^{(2)}n\bigr)\\
&=\left(\frac{\exp((1+\eps)\log^{(3)} n)}{\log n}\right)^{\frac{m-1}{3}}=\left(\frac{(\log^{(2)}n)^{1+\eps}}{\log n}\right)^{\frac{m-1}{3}}.
\end{align*}
\end{proof}

\subsection{Maximal recursive trees in $G_1^n$.} By  Lemma \ref{exLnm=} and Theorem \ref{Ln} whp there are $(1/2+o(1))\log n$ vertices with 
a loop. Clearly each of these vertices is a root of an isolated tree component, the maximum size recursive tree formed by root's descendants in $G_1^n$.  Vertex $r$ is a root with probability $(2r-1)^{-1}$. {\it Conditioned on $r$ being a root\/},  let 
$X(t;r)$, $(t\ge r)$, denote the size of the isolated tree component of $G_1^t$ rooted at $r$. Since $r$ is fixed, we will simply write $X(t)$ instead of $X(t;r)$
in the derivations below. In particular, $X(n)$ is the size of the the tree component in $G_1^n$ rooted at $r$. The total vertex degree of the component is 
$2X(t)$. By the definition of $G_1^{t+1}$, we have: for $k\ge 1$, and $t\ge r$,
\begin{equation}\label{1}
\begin{aligned}
\Bbb E[X^k(t+1) \boldsymbol |\,G_1^t]&=(X(t)+1)^k\,\frac{2X(t)}{2t+1}+X^k(t)\left(1-\frac{2X(t)}{2t+1}\right)\\
&=\frac{2X(t)}{2t+1}\sum_{j=0}^k\binom{k}{j} X^j(t)+X^k(t)\left(1-\frac{2X(t)}{2t+1}\right)\\
&=X^k(t)+\frac{2X(t)}{2t+1}\sum_{j=0}^{k-1}\binom{k}{j} X^j(t)\\
&=X^k(t)\frac{2(t+k)+1}{2t+1}+\frac{2}{2t+1}\sum_{j=0}^{k-2}\binom{k}{j} X^{j+1}(t).\\
\end{aligned}
\end{equation}
\begin{lemma}\label{mart} Using notation $x^{(\ell)}=x(x+1)\cdots (x+\ell-1)$, we have
\[
\Bbb E\bigl[X^{(\ell)}(t+1)\boldsymbol|G_1^t\bigr]=\frac{2(t+\ell)+1}{2t+1} X^{(\ell)}(t).
\]
Consequently $X^{(\ell)}(t)\prod_{j=0}^{\ell-1}\,\bigl(2(t+j)+1\bigr)^{-1}$, $(t\ge r)$, is a martingale.
\end{lemma}
{\bf Note.\/} That same function $x^{(\ell)}$ had been used to construct a martingale for $D(r,t)$, $(t\ge r$),  the degree of vertex $r$ in $G_m^t$, see
\cite{Bol1}, and in the $\beta$-extension of $G_1^t$, see \cite{Mor0}, \cite{Mor}. Understandably, our argument is quite different. 
\begin{proof} Recall first that
\[
x^{(\ell)}=\sum_{k=1}^{\ell} x^k s(\ell,k),
\]
where $s(\ell,k)$ is the signless, first-kind, Stirling number, i.e. the number of permutations of the set $[\ell]$ with $k$ cycles. In particular,
\begin{equation}\label{Com}
\sum_{\ell\ge 1}\eta^{\ell}\frac{s(\ell,k)}{\ell!}=\frac{1}{k!}\log^k\frac{1}{1-\eta},\quad |\eta|<1,
\end{equation}
Comtet \cite{Com}, Section 5.5. Using \eqref{1}, we have
\begin{align*}
(2t+1)&\Bbb E\bigl[X^{(\ell)}(t+1)\boldsymbol|G_1^t\bigr]=\sum_{k=1}^{\ell}s(\ell,k)\cdot
\Biggl(\!\!\bigl(2(t+k)+1\bigr) X^k(t)\\
&\quad+2\sum_{j=0}^{k-2}\binom{k}{j}X^{j+1}(t)\!\Biggr)=:\sum_{k=1}^{\ell} \sigma(\ell,k)X^k(t),\\
\sigma(\ell,k)&\!=\!\left\{\!\!\begin{aligned}
&\bigl(2(t+\ell)+1\bigr)s(\ell,\ell),&&\text{if }k=\ell,\\
&\bigl(2(t+k)+1\bigr)s(\ell,k)+2\sum_{j=k+1}^{\ell}\!\!s(\ell,j)\binom{j}{k-1},&&\text{if }k<\ell.\end{aligned}\right.
\end{align*}
We need to show that $\sigma(\ell,k)=\bigl(2(t+\ell)+1\bigr)s(\ell,k)$ for $k<\ell$, which is equivalent to
\[
ks(\ell,k)+\!\!\sum_{j=k+1}^{\ell}\!\!s(\ell,j)\binom{j}{k-1}\!=\!\ell s(\ell,k)\Longleftrightarrow \sum_{j=k}^{\ell}s(\ell,j)\binom{j}{k-1}=\ell s(\ell,k).
\]
To prove the latter identity, it suffices to show that, for a fixed $k$, the exponential generating functions of the two sides coincide. By \eqref{Com},
\begin{align*}
&\quad\sum_{\ell\ge 1}\frac{\eta^{\ell}}{\ell!}\sum_{j=k}^{\ell}s(\ell,j)\binom{j}{k-1}=\sum_{j\ge k}\binom{j}{k-1}\sum_{\ell\ge j}\frac{\eta^{\ell}}{\ell!} s(\ell,j)
\\
&=\sum_{j\ge k}\binom{j}{k-1} \frac{1}{j!}\log^j\frac{1}{1-\eta}=\frac{1}{(k-1)!}\left(\log^{k-1}\frac{1}{1-\eta}\right)\sum_{s\ge 1}\frac{1}{s!}\log^s\frac{1}{1-\eta}\\
&=\frac{1}{(k-1)!}\left(\log^{k-1}\frac{1}{1-\eta}\right)\left(\frac{1}{1-\eta}-1\right)=\frac{1}{(k-1)!}\left(\log^{k-1}\frac{1}{1-\eta}\right)\frac{\eta}{1-\eta}.
\end{align*}
And, using \eqref{Com} again,
\begin{align*}
&\sum_{\ell\ge 1}\frac{\eta^{\ell}}{\ell!}\,\ell s(\ell,k)=\eta\sum_{\ell\ge 1}\frac{\ell\eta^{\ell-1}}{\ell!}\,s(\ell,k)\\
&=\eta\frac{d}{d\eta}\Biggl(\frac{1}{k!}\log^k\frac{1}{1-\eta}\Biggr)=\frac{1}{(k-1)!}\left(\log^{k-1}\frac{1}{1-\eta}\right)\frac{\eta}{1-\eta}.
\end{align*}
\end{proof}
Since $X^{(\ell)}(t)\prod_{j=0}^{\ell-1}\bigl(2(t+j)+1\bigr)^{-1}$ is a martingale, we have
\[
\Bbb E\Bigl[X^{(\ell)}(t)\prod_{j=0}^{\ell-1}\bigl(2(t+j)+1\bigr)^{-1}\Bigr]=\ell!\prod_{j=0}^{\ell-1}\bigl(2(r+j)+1\bigr)^{-1},
\]
or equivalently
\[
\Bbb E\Biggl[\binom{X(t)+\ell-1}{\ell}\Biggr]=\prod_{j=0}^{\ell-1}\frac{2(t+j)+1}{2(r+j)+1},\quad t\ge r.
\]
So, using $X(t)\le t$, we conclude that
\begin{equation}\label{density2}
\lim_{n\to\infty} \Bbb E\left[\left(\frac{X(n)}{n}\right)^{\ell}\right]=2^{\ell}\ell!\prod_{j=0}^{\ell-1}\frac{1}{2(r+j)+1}.
\end{equation}
Therefore $n^{-1}X(n)$ converges in distribution, and with all its moments, to a random variable $Z_2$ whose moments $\Bbb E[Z_2^{\ell}]$ are given by the RHS of the equation above.  (Subindex $2$ stands for the initial degree $2$ of the root $r$.)

Recall that $X(n)$ is the size of the maximal recursive tree rooted at $r$, conditioned on $r$ looping back on itself, which happens with probability
$(2r-1)^{-1}$. Let us determine the distribution of $X(n)$ conditioned on $r$ attaching itself to one of the vertices in $[r-1]$, which happens with 
probability $2(r-1)/(2r-1)$. Introduce $Y(t)=X(t)-1/2$, so that $Y(r)=1/2$. The total vertex degree of the maximal recursive tree in $G_1^t$ rooted at $r$
is $2X(t)-1=2Y(t)$. Therefore for $t\ge r$ we have 
\[
\Bbb E[Y^k(t+1) \boldsymbol |\,G_1^t]=(Y(t)+1)^k\,\frac{2Y(t)}{2t+1}+Y^k(t)\left(1-\frac{2Y(t)}{2t+1}\right),
\]
which is the top equation in \eqref{1} with $Y$ substituting for $X$. Therefore in the notations of Lemma \ref{mart} we have
\begin{lemma}\label{mart1} $Y^{(\ell)}(t)\prod_{j=0}^{\ell-1}\,\bigl(2(t+j)+1\bigr)^{-1}$, $(t\ge r)$, is a martingale.
\end{lemma}
Consequently
\[
\Bbb E\Bigl[Y^{(\ell)}(t)\prod_{j=0}^{\ell-1}\bigl(2(t+j)+1\bigr)^{-1}\Bigr]=\left(\frac{1}{2}\right)^{(\ell)}\,\prod_{j=0}^{\ell-1}\bigl(2(r+j)+1\bigr)^{-1},
\]
where
\[
Y^{(\ell)}(t)=\prod_{j=0}^{\ell-1}(X(t)-1/2-j),\quad \left(\frac{1}{2}\right)^{(\ell)}=2^{-\ell}(2\ell-1)!!.
\]
Therefore
\[
\Bbb E\Bigl[\prod_{j=0}^{\ell-1}(X(t)-1/2-j)\Bigr]=2^{-\ell}(2\ell-1)!!\prod_{j=0}^{\ell-1}\frac{2(t+j)+1}{2(r+j)+1},\quad t\ge r.
\]
We conclude that 
\begin{equation}\label{density1}
\lim_{n\to\infty} \Bbb E\left[\left(\frac{X(n)}{n}\right)^{\ell}\right]=(2\ell-1)!!\prod_{j=0}^{\ell-1}\frac{1}{2(r+j)+1}.
\end{equation}
Therefore $n^{-1}X(n)$ converges in distribution, and with all its moments, to a random variable $Z_1$ whose moments $\Bbb E[Z_1^{\ell}]$ are given by the RHS of the equation above.  To identify the limiting distributions, recall that the classic beta probability distribution has density 
\[
f(x;\a,\be)=\frac{\Gamma(\a+\be)}{\Gamma(\a)\Gamma(\be)}x^{\a-1}(1-x)^{\be-1},\quad x\in (0,1),
\]
parametrized by two parameters $\a>0$, $\be>0$, and moments
\begin{equation}\label{moments}
\int_0^1 x^{\ell} f(x;\a,\be)\,dx=\prod_{j=0}^{\ell-1}\frac{\a+j}{\a+\be+j}.
\end{equation}
Comparing the RHS of \eqref{moments} with the RHSides of \eqref{density2} and of \eqref{density1}, we see that $Z_2$ is beta-distributed
with parameters $\a=1$, $\be=r-1/2$, and $Z_1$ is beta-distributed with parameters $\a=1/2$, $\beta=r$. We have proved

\begin{theorem}\label{mixture} The limiting distribution of scaled size of the max-tree rooted at vertex $r$ is the mixture of two beta-distributions, with parameters
$\a=1$, $\be=r-1/2$ and $\a=1/2$, $\beta=r$, weighted by $\frac{1}{2r-1}$ and $\frac{2(r-1)}{2r-1}$ respectively.
\end{theorem}
\bigskip

{\bf Note.\/} $\{G_1^t\}$ is a special case of the graph process $\{G_{1,\delta}^t\}$, $\delta\ge -1$, see \cite{Hof}. Like $G_1^1$, $G_{1,\delta}^1$ consists of a single vertex $1$ with a single self-loop. Recursively, conditioned on $G_{1,\delta}^t$, the new vertex $t+1$ forms a self-loop with probability
$\frac{1+\delta}{t(2+\delta)+(1+\delta)}$, and attaches itself to a vertex $i\in [t]$ with probability proportional to the degree of $i$ in $G_{1,\delta}^t$ plus
$\delta$. Only minor modifications are needed to prove the following $\delta$-extension of Theorem \ref{mixture}.
\begin{theorem}\label{mix,ext} The limiting distribution of scaled size of the max-tree rooted at vertex $r$ is the mixture of two beta-distributions, with parameters
$\a=1$, $\be=r-1+\frac{1+\delta}{2+\delta}$ and $\a=\frac{1+\delta}{2+\delta}$, $\beta=r$, weighted by $\frac{1+\delta}{(2+\delta)r-1}$ and 
$\frac{(2+\delta)(r-1)}{(2+\delta)r-1}$ respectively.
\end{theorem}
\bigskip

{\bf Acknowledgment.\/} I am grateful to  Huseyin Acan for many discussions in the course of our
joint research on the random $m$-out, age-biased graph \cite{AcaPit}, broadly close to the preferential attachment graph.
I thank the participants of Combinatorics/Probability seminar (Ohio State Univ.)  for encouraging discussion of this study. The last subsection provides  an answer to a question asked by David Sivakoff.

\end{document}